\documentclass[12pt]{amsart}
\usepackage{graphicx}
\usepackage{graphicx}
\usepackage{epsfig}
\usepackage{pstricks}

\newtheorem{thm}{Theorem}[section]
\newtheorem{prop}[thm]{Proposition}
\newtheorem{defn}[thm]{Definition}
\newtheorem{lemma}[thm]{Lemma}

\newtheorem{cor}[thm]{Corollary}
\newtheorem{remark}[thm]{Remark}
\newtheorem{example}[thm]{Example}

\usepackage{amsmath}
\usepackage{amsxtra}
\usepackage{amscd}
\usepackage{amsthm}
\usepackage{amsfonts}
\usepackage{amssymb}
\usepackage{eucal}
\newcommand{\bmb}{\left( \begin{array}{rr}}
\newcommand{\enm}{\end{array}\right)}
\newcommand{\C}{{\mathbb C}}
\newcommand{\Z}{{\mathbb Z}}

\newcommand{\al}{{\alpha}}

\numberwithin{equation}{section}

\begin{document}
\title{An inhomogeneous Lambda-determinant}
\author{Philippe Di Francesco} 
\address{
%Department of Mathematics, University of Michigan,
%530 Church Street, Ann Arbor, MI 48190, USA
%and 
Institut de Physique Th\'eorique du Commissariat \`a l'Energie Atomique, 
Unit\'e de Recherche associ\'ee du CNRS,
CEA Saclay/IPhT/Bat 774, F-91191 Gif sur Yvette Cedex, 
FRANCE. e-mail: philippe.di-francesco@cea.fr}

\begin{abstract}
We introduce a multi-parameter generalization of the Lambda-determinant
of Robbins and Rumsey, based on the cluster algebra with coefficients
attached to a $T$-system recurrence. We express the result as a weighted 
sum over alternating sign matrices.
\end{abstract}

\maketitle
\date{\today}
\tableofcontents

\section{Introduction}

The so-called Lambda-determinant was introduced by Robbins and Rumsey \cite{RR}
as a natural generalization of the ordinary determinant, via a one-parameter deformation
of the Dodgson condensation algorithm \cite{DOD} that expresses the determinant
of any $n\times n$ matrix in terms of minors of sizes $n-1$ and $n-2$.
This produces, for each $n\times n$ matrix $A$,
a Laurent polynomial of its entries, involving only monomials with powers of $\pm 1$, coded
by Alternating Sign Matrices (ASM) of same size $n$. Robbins and Rumsey were able to write a compact
formula for the Lambda-determinant of any matrix $A$, as a sum over the ASMs of same size with
explicit coefficients \cite{RR} (see also \cite{Bressoud} for a lively account of the discovery of ASMs).
ASMs are known to be in bijection with configurations
of the Six Vertex (6V) model on a square $n\times n$ grid, with special Domain Wall Boundary Conditions (DWBC)
\cite{KUP}.
The latter are obtained by choosing an orientation of the edges of the underlying square lattice,
such that each vertex of the grid has two incoming and two outgoing adjacent edges. The resulting
six local configurations read as follows:
\begin{equation}\label{sixvertices} \raisebox{-1.cm}{\hbox{\epsfxsize=14.cm \epsfbox{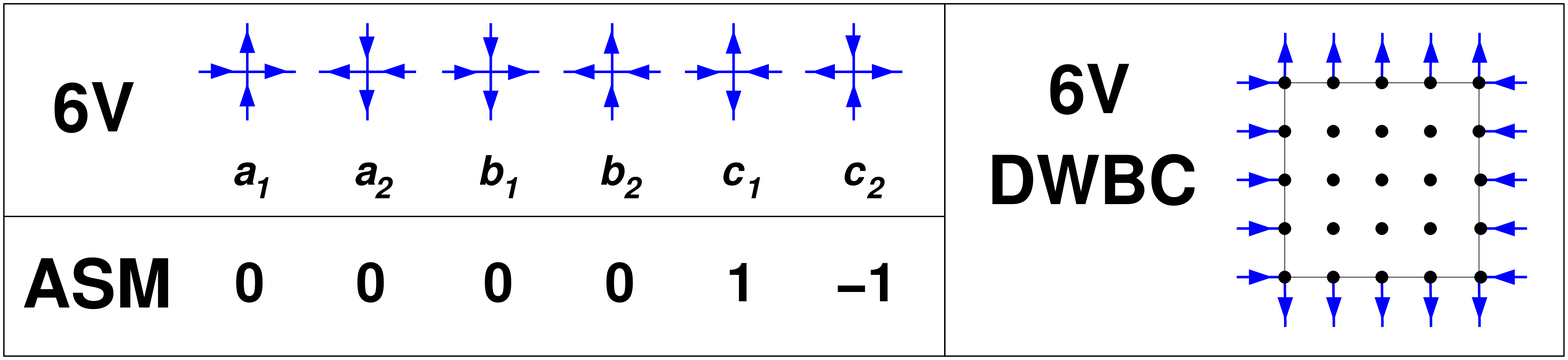}}} 
\end{equation}
The boundary condition (DWBC)
imposes that horizontal external edges point towards the grid, while vertical external edges point out of the grid.
The bijection between the 6V configurations and ASMs goes as follows: 
the $a$ and $b$ type vertices correspond to the entries $0$, while the $c_1$ type vertex 
corresponds to the entry $1$, and the $c_2$ type vertex to the entry $-1$. Conversely, for any $n\times n$ ASM $B$,
there is a unique 6V-DWBC configuration on the $n\times n$ square grid.
For any fixed ASM $B$, we denote by ${\mathcal C}(B)_{i,j}$ the configuration of the vertex $(i,j)$ in the 
corresponding 6V-DWBC model,  with ${\mathcal C}(B)_{i,j}\in \{a_1,a_2,b_1,b_2,c_1,c_2\}$.

In this note, we define an inhomogeneous, multi-parameter generalization of the
Lambda-determinant of an $n\times n$ matrix $A$.
This is done by deforming the Dodgson condensation algorithm \cite{DOD} by use of two sets of parameters
$\lambda_a,\mu_a$, $a\in \Z$.  More precisely:

\begin{defn}\label{onedef}
Let $A=(a_{i,j})_{i,j\in I}$, $i=\{1,2,...,n\}$, and sequences $\lambda=(\lambda_a)_{a\in \Z}$,
$\mu=(\mu_a)_{a\in\Z}$ of fixed parameters.
The generalized Lambda-determinant of $A$, denoted $\vert A\vert_{\lambda,\mu}$ is defined inductively
by the following modified Dodgson condensation algorithm. Let $A_{i_1,...,i_k}^{j_1,..,j_k}$ denote the submatrix of $A$
obtained by erasing rows $i_1,...,i_k$ and columns $j_1,...,j_k$, and $s$ be the shift operator acting on sequences as
$(s\lambda)_a=\lambda_{a+1}$ and $(s^{-1}\lambda)_a=\lambda_{a-1}$ and similarly on $\mu$. We have:
$$\vert A\vert_{\lambda,\mu} =\frac{\mu_0\, \vert A_1^1\vert_{s\lambda,s\mu}\, \vert A_n^n\vert_{s^{-1}\lambda,s^{-1}\mu}
+\lambda_0\, \vert A_1^n\vert_{s^{-1}\lambda,s\mu}\,
\vert A_n^1\vert_{s\lambda,s^{-1}\mu} }{\vert A_{1,n}^{1,n}\vert_{\lambda,\mu}}
$$
This recursion relation on the size of the matrix, together with the initial condition that a $0\times0$ matrix has
generalized Lambda-determinant 1 and a $1\times 1$ matrix $(a)$ has 
generalized Lambda-determinant $a$ determines $\vert A\vert_{\lambda,\mu}$
completely.
\end{defn}

This generalization of the Lambda-determinant is connected to the expressions found in \cite{SPY} 
for solutions of the octahedron recurrence as partition functions of domino tilings of Aztec diamonds.
%As the definition involves divisions, the Lambda-determinant may not always be well-defined. 
%As the definition involves divisions, the (generalized) Lambda-determinant may not always be well-defined. 
The main result of the present paper is the following closed formula, which holds whenever the definition makes sense:

\begin{thm}\label{main}
For fixed parameters $\lambda_a,\mu_a$, $a\in \{-n+2,-n+3,...,n-3,n-2\}$ the generalized Lambda-determinant
of a matrix $A=(a_{i,j})_{i,j\in I}$, $I=\{1,2,...,n\}$, is expressed as follows:
\begin{equation}\label{formu}
|A|_{\lambda;\mu}=\sum_{ASM\, B} \, \prod_{i,j=1}^n w_{i,j}(A,B;\lambda,\mu)
\end{equation}
where the sum extends over all the $n\times n$ Alternating Sign Matrices $B=(b_{i,j})_{i,j\in I}$, and the weights 
$w_{i,j}(A,B,\lambda;\mu)$ are defined as:
$$ w_{i,j}(A,B;\lambda,\mu)= (a_{i,j})^{b_{i,j}}\times  \left\{ \begin{matrix}
\lambda_{j-i} & {\rm if} \ {\mathcal C}(B)_{i,j}=a_1 \\ 
\mu_{i+j-n-1} & {\rm if}\  {\mathcal C}(B)_{i,j}=b_1 \\
\big( \lambda_{j-i}+\mu_{i+j-n-1} \big) & {\rm if}\  {\mathcal C}(B)_{i,j}=c_2 \\
1 & {\rm otherwise} 
\end{matrix} \right. \, .$$
\end{thm}

The paper is organized as follows. In Section 2, we recall Robbins and Rumsey's original definition of the Lambda-determinant,
and express it as a solution of the $T$-system/octahedron recurrence with one coefficient. In Section 3, we 
define the generalized Lambda-determinant as a solution of a $T$-system with inhomogeneous coefficients. The latter is
shown to be part of a cluster algebra of infinite rank, and as such to possess the Laurent property. 
Using the T-system relation, we present a non-homogeneous example leading to simple product formulas for the
generalized Lambda-determinant of the matrix with entries $a_{i,j}=1$ for all $i,j$.
Section 4 is devoted
to the actual computation of the generalized Lambda-determinant, by solving the $T$-system with inhomogeneous coefficients.
This is done via a matrix representation of the $T$-system relation that leads to a generalization of the solution of \cite{DFK12}.
In Section 5 the solution is rephrased in terms of networks, namely of paths on directed graphs. The latter are finally
mapped onto configurations of the 6V-DWBC model, leading to the proof of Theorem \ref{main}. Section 6 is devoted
to further properties of the generalized Lambda-determinant, an expression purely in terms of ASMs,
some explicit examples, and the complete solution
of the $T$-system with inhomogeneous coefficients, and finally some concluding remarks.

\medskip
\noindent{\bf Acknowledgments.} 
We thank R. Kedem for many useful discussions.
We also thank the Mathematical Science Research Institute in Berkeley, CA 
and the organizers of the semester ``Cluster Algebras" for hospitality
while this work was completed. This work is partially supported
by the CNRS PICS program.

\section{The classical Lambda-determinant}

\subsection{Desnanot-Jacobi identity and determinant}

Let $I=\{1,2,...,n\}$. For any matrix $M=(m_{i,j})_{i,j\in I}$ we use the same notations as in Def.\ref{onedef}.
For any given $k+1\times k+1$ matrix $M$, we have the celebrated identity:
\begin{equation}\label{desnajac}
\vert M\vert \times \vert M_{1,n}^{1,n}\vert=\vert M_n^n\vert \times \vert M_1^1\vert -
\vert M_1^n\vert \times \vert M_n^1\vert
\end{equation}
This gives rise to the Dodgson condensation algorithm for effective computing of
determinants, as the identity may be used as a closed recursion relation on the size of the matrix,
allowing for computing its determinant from the initial data of determinants of matrices of 
size 0 and 1 (equal respectively to 1
and the single matrix element).

More formally, we may recast the algorithm using the so-called $A_\infty$ $T$-system relation:
\begin{equation}\label{tsys}
T_{i,j,k+1}T_{i,j,k-1}=T_{i,j+1,k}T_{i,j-1,k}-T_{i+1,j,k}T_{i-1,j,k}
\end{equation}
for any $i,j,k\in \Z$ with fixed parity of $i+j+k$. 

Let $A=(a_{i,j})_{i,j\in I}$ be a fixed $n\times n$ matrix. Together with the initial data:
\begin{eqnarray}
T_{\ell,m,0}&=&1\quad \qquad \qquad\qquad (\ell,m\in \Z;\ell+m=n\, {\rm mod}\, 2)\nonumber \\ 
T_{i,j,1}&=&a_{{j-i+n+1\over 2},{i+j+n+1\over 2}} \quad
(i,j\in\Z;i+j=n+1\, {\rm mod}\, 2; |i|+|j|\leq n-1)\, ,\label{initdat}
\end{eqnarray}
the solution of the $T$-system \eqref{tsys} satisfies:
\begin{equation}
T_{0,0,n}=\det(A)
\end{equation}

\subsection{$T$-system with a coefficient and Lambda-determinant}

Robbins and Rumsey have introduced the Lambda-determinant by applying the
Dodgson algorithm to the following modified Desnanot-Jacobi identity. Let $\lambda\in \C^*$
be a fixed parameter and $A=(a_{i,j})_{i,j\in I}$ be a fixed $n\times n$ matrix. 

\begin{defn}\label{lddef}
The Lambda-determinant of $A$, denoted $\vert A\vert_{\lambda}$ is defined as the solution
$\vert A\vert_{\lambda}=T_{0,0,n}$ of the following $T$-system with a coefficient
\begin{equation}\label{coeftsys}
T_{i,j,k+1}T_{i,j,k-1}=T_{i,j+1,k}T_{i,j-1,k}+\lambda \,T_{i+1,j,k}T_{i-1,j,k}
\end{equation}
and subject to the initial conditions \eqref{initdat}.
\end{defn}

Note that for $\lambda=-1$, we recover the usual determinant: $\vert A\vert_{-1}=\det(A)$.

\begin{example} 
The generalized Lambda-determinant of a $3\times 3$ matrix $A$ reads:
\begin{equation*}
\left\vert \begin{matrix} a & b & c \\ d & e & f\\ g & h & k \end{matrix} \right\vert_{\lambda}
= \lambda^3 \, c e g+\lambda^2 \, c d h+\lambda^2 \, b f g 
+ \, a e k+ \lambda \, b d k+\lambda \, a f h+\lambda
(\lambda+1) \, {b d f h \over e}
\end{equation*}
\end{example}

\begin{example}
The Lambda-determinant of the Vandermonde matrix $A=(a_i^{j-1})_{i,j\in I}$ reads
$$ \vert A\vert_{\lambda}=\prod_{1\leq i<j \leq n} \big( a_j+\lambda a_i\big) $$
\end{example}

\subsection{The Robbins-Rumsey formula and Alternating Sign Matrices}

A remarkable property of the definition of the Lambda-determinant, is that it produces a
Laurent polynomial of the matrix elements $a_{i,j}$, which is a sum over monomials
with only powers $\pm 1$. These powers are coded by $n\times n$ so-called 
Alternating Sign Matrices (ASM), namely matrices $B=(b_{i,j})_{i,j\in I}$ with
entries $b_{i,j}\in \{0,1,-1\}$, with non-negative partial row sums:
$\sum_{j=1}^k b_{i,j}\geq 0$, $k=1,2,...,n-1$, $i\in I$; and with row sums equal to one:
$\sum_{j=1}^n b_{i,j}=1$, $i\in I$. More precisely, Robbins and Rumsey found the following
formula for the Lambda-determinant:
\begin{equation}\label{ldRR}
\vert A\vert_{\lambda}=\sum_{ASM\, B} \lambda^{{\rm Inv}(B)} (1+\lambda^{-1})^{\#(-1)_B} \prod_{i,j\in I}a_{i,j}^{b_{i,j}}
\end{equation}
where  $\#(-1)_B$ is the total number of entries in $B$ that are equal to $-1$, and
$${\rm Inv}(B)=\sum_{k<\ell\atop m<p} b_{k,p}b_{\ell,m} $$
is the generalized inversion number of $B$.

\begin{remark}
Note that for $\lambda=-1$ the sum truncates to only the contribution of ASMs with no $-1$ entry, which
are the permutation matrices $P$, and ${\rm Inv}(P)$ is the usual inversion number of the permutation matrix $P$,
so that \eqref{ldRR} reduces to the usual determinant formula.
\end{remark}

\begin{remark}
From eq.\eqref{ldRR}, we see that the Lambda-determinant is well defined for any matrix $A$ with non-vanishing
entries. More precisely, as the matrix elements $b_{i,j}=-1$ of ASMs may only occur for $2\leq i,j\leq n-1$,
the Lambda-determinant is well-defined for any $n\times n$ matrix $A$ such that $a_{i,j}\neq 0$ for all $2\leq i,j\leq n-1$.
\end{remark}

\section{An inhomogeneous generalization of the Lambda-determinant}

\subsection{$T$-system and Cluster Algebra}

The unrestricted $A_\infty$ $T$-system \eqref{tsys} is known to be part of an infinite rank cluster algebra
with coefficients \cite{FZI}.
The T-system in this form (with $T_{i,j,k}$
defined only for $i+j+k=1$ mod 2), was shown to be a particular mutation in an infinite rank cluster algebra 
with coefficients,
with cluster variables of the form $(T_{i,j,k-1},T_{i',j',k})_{i,j,i',j'\in \Z}$ and non-vanishing fixed 
coefficients $(\lambda_i)_{i\in \Z}$ \cite{DFK08}.
It reads:
\begin{equation}
T_{i,j,k+1}T_{i,j,k-1}=T_{i,j+1,k}T_{i,j-1,k}+\lambda_i \,T_{i+1,j,k}T_{i-1,j,k}
\end{equation}
As a consequence of the cluster algebra Laurent property \cite{FZI}, the solution $T_{i,j,k}$ of this equation
may be expressed as a Laurent polynomial of any choice of admissible initial values, with coefficients
in $\Z[(\lambda_i)_{i\in\Z}]$. The ``coefficients" $\lambda_i$ are inhomogeneous, as they now depend
on the value of $i$. In particular, choosing homogeneous coefficients $\lambda_i=\lambda$ independent of $i$,
we recover the Laurent property observed for the Lambda-determinant of Robbins and Rumsey.

As it turns out, we may consider an even more general inhomogeneous coefficient $T$-system equation:
\begin{equation}\label{inhomtsys}
T_{i,j,k+1}T_{i,j,k-1}=\mu_j\, T_{i,j+1,k}T_{i,j-1,k}+\lambda_i \,T_{i+1,j,k}T_{i-1,j,k}
\end{equation}
for any fixed non-vanishing coefficients $(\lambda_i)_{i\in\Z}$ and $(\mu_j)_{j\in\Z}$. We have the following

\begin{thm}\label{inhomclust}
The equation \eqref{inhomtsys} is a particular mutation in a cluster algebra $\mathcal A$ with coefficents. The initial seed
is given by the cluster $X_0=\big((T_{i,j,0})_{i,j\in\Z},(T_{i',j',1})_{i',j'\in \Z}\big)$ 
with $i+j=1$ mod 2 and $i'+j'=0$ mod 2 as usual,
and the coefficients $\big((\lambda_a)_{a\in\Z},(\mu_b)_{b\in \Z}\big)$, and the
extended exchange matrix ${\tilde B}_0$ with the following infinite
block form: ${\tilde B}_0=\begin{pmatrix} 0 & B\\ -B & 0 \\ L_0 & L_1 \\ M_0 & M_1 \end{pmatrix}$, where:
\begin{eqnarray*}(B)_{i,j;i',j'}&=&\delta_{i,i'}\delta_{|j-j'|,1}-\delta_{j,j'}\delta_{|i-i'|,1}\\
(L_0)_{a;i,j}&=&\delta_{i,a} \qquad \qquad \qquad (L_1)_{a;i',j'}=-\delta_{i',a}\\
(M_0)_{b;i,j}&=&-\delta_{j,b} \qquad \qquad \qquad (M_1)_{b;i',j'}= \delta_{j',b}
\end{eqnarray*}
\end{thm}
\begin{proof}
We must show that the cluster $X_k=\big((T_{i,j,k})_{i,j\in\Z},(T_{i',j',k+1})_{i',j'\in \Z}\big)$, where $T_{i,j,k}$
is the solution of \eqref{inhomtsys} with prescribed values of $\big((T_{i,j,0})_{i,j\in\Z},(T_{i',j',1})_{i',j'\in \Z}\big)$, 
is a cluster in the cluster algebra $\mathcal A$. First let us represent the quiver ${\tilde Q}_0$ coded by the extended exchange
matrix ${\tilde B}_0$. We represent by a filled $\bullet$  (resp. empty $\circ$) circle the vertices $(i,j)$ such that $i+j=1$ mod 2 
(resp. $(i',j')$ such that $i'+j'=0$ mod 2),
and by a cross $\otimes$ (resp. square $\square$) the vertices $(a)$ indexing $\lambda_a$ (resp. (b) indexing $\mu_b$) 
of ${\tilde Q}_0$. We have the following local structure,
around respectively vertices $(i,j)$ $\bullet$ and $(i',j')$ $\circ$:
$$ \raisebox{-1.cm}{\hbox{\epsfxsize=11.cm \epsfbox{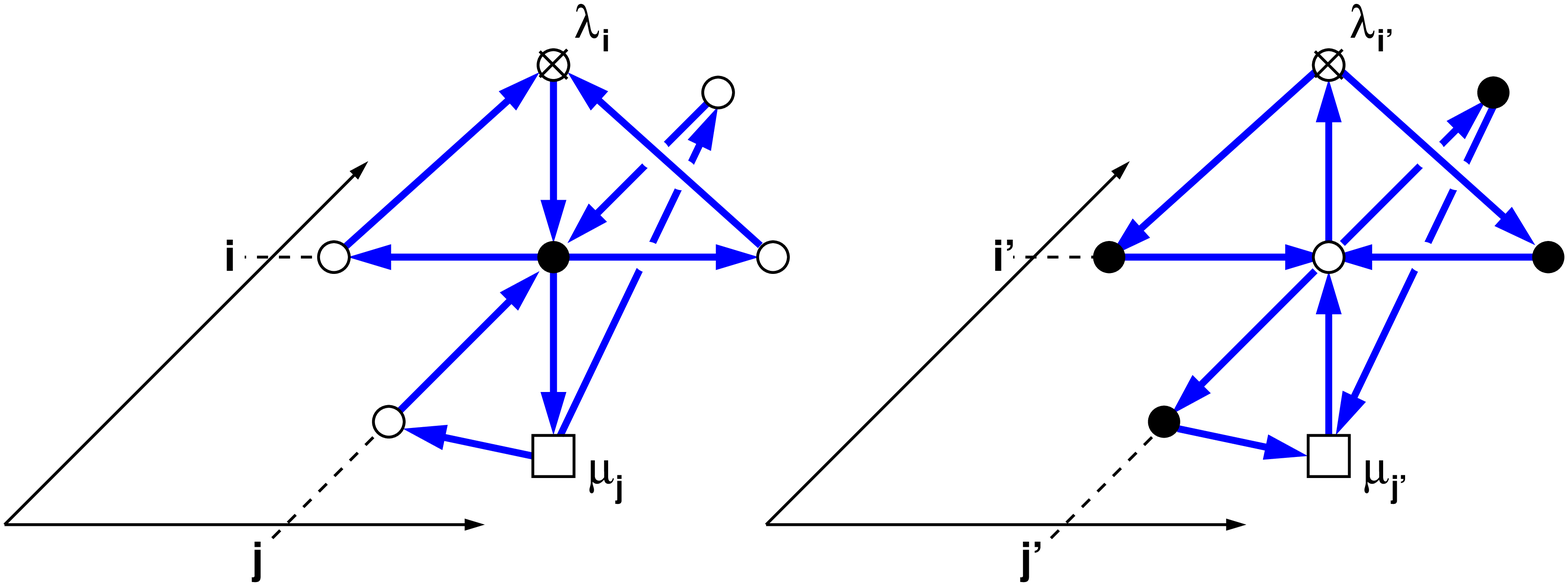}}} $$
Note that the vertices $(a=i)$ indexing $\lambda_i$ are connected to all $(i,j)$, $j\in \Z$
while vertices $(b=j)$ indexing $\mu_j$ are connected to all $(i,j)$, $i\in \Z$, with edges of alternating orientations.
We may reach the cluster $X_1$ from the initial cluster $X_0$ by a compound mutation, obtained by
mutating all the vertices $\bullet$. The resulting quiver ${\tilde Q}_1$ is identical to ${\tilde Q}_0$
with all edge orientations flipped, namely ${\tilde B}_1= -{\tilde B}_0$, as illustrated below
for mutations in a quadrant around a $\bullet$ vertex:
$$ \raisebox{-1.cm}{\hbox{\epsfxsize=16.cm \epsfbox{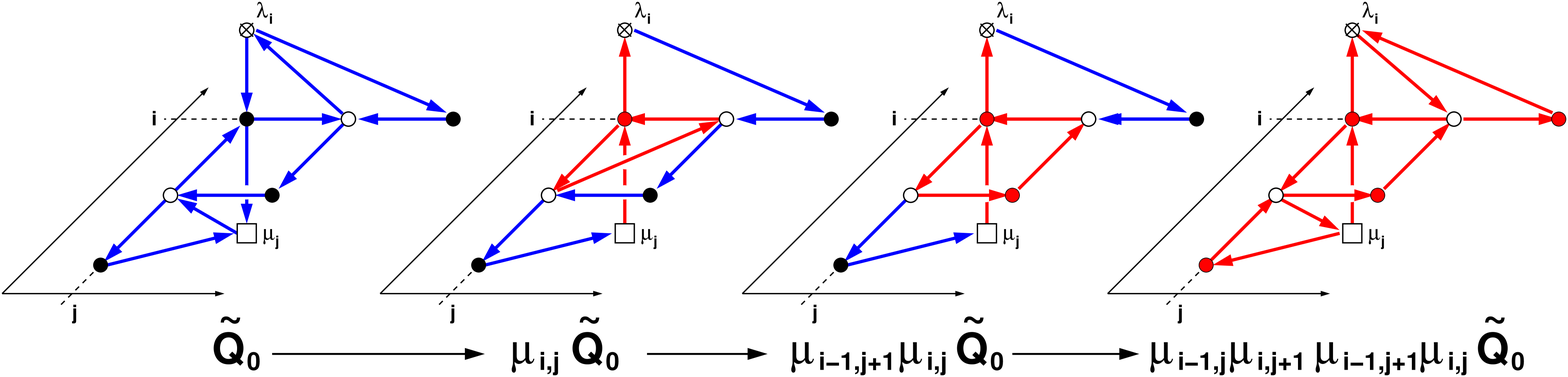}}}  $$
(Recall that a mutation at a vertex $v$ of a quiver flips the orientations of all incident edges, creates ``shortcut"
edges $u\to w$ for any length $2$ path $u\to v\to w$ before mutation, and any 2-loop $v\to w\to v$ thus
created must be eliminated.)

The cluster $X_2$ is then reached by mutating all vertices $\circ$, and has the exchange matrix
${\tilde B}_2={\tilde B}_0$, etc.
\end{proof}

As a consequence of the theorem, we have the Laurent property: the solution $T_{i,j,k}$
of the system \eqref{inhomtsys} for prescribed values of 
$\big((T_{i,j,0})_{i,j\in\Z},(T_{i',j',1})_{i',j'\in \Z}\big)$, is a Laurent polynomial of these values,
with coefficients in $\Z[(\lambda_a)_{a\in\Z},(\mu_b)_{b\in \Z}]$. 

\subsection{Generalized Lambda-determinant: definition and properties}

In view of the results of previous section, we may recast the Definition \ref{onedef} into the following:
\begin{lemma}
Given a square matrix $A=(a_{i,j})_{i,j\in I}$, and some parameters $\lambda_{a},\mu_a$,
$a\in J=\{-(n-2),-(n-3),...,-1,0,1,...,n-3,n-2\}$,
the inhomogeneous Lambda-determinant of $A$ is the solution $T_{0,0,n}$
of the inhomogeneous $T$-system \eqref{inhomtsys}, with initial conditions \eqref{initdat}.
\end{lemma}

As a consequence of Theorem \ref{inhomclust}, we know that $|A|_{\lambda;\mu}$ is a Laurent polynomial of the entries
$a_{i,j}$ of $A$, with coefficients in $\Z[(\lambda_a)_{a\in J},(\mu_a)_{a\in J}]$.

\begin{example} 
The generalized Lambda-determinant of a $3\times 3$ matrix $A$ reads:
\begin{eqnarray*}
\left\vert \begin{matrix} a & b & c \\ d & e & f\\ g & h & k \end{matrix} \right\vert_{\lambda,\mu}
&=& \lambda_1 \lambda_0\lambda_{-1} \, c e g+\lambda_0\mu_0\lambda_1 \, c d h+\lambda_0\mu_0\lambda_{-1} \, b f g \\
&&+\mu_1\mu_0\mu_{-1} \, a e k+ \lambda_0\mu_0 \mu_1 \, b d k+\lambda_0\mu_0 \mu_{-1} \, a f h+\lambda_0\mu_0 
(\lambda_0+\mu_0) \, {b d f h \over e}
\end{eqnarray*}
\end{example}

\begin{example}\label{maxvan}
Let $\lambda_a=\lambda$ and $\mu_a=\mu$, for all $a$, and $A$ be the Vandermonde matrix
$A=(a_i^{j-1})_{i,j\in I}$. Then we have:
$$ |A|_{\lambda;\mu}=\prod_{1\leq i<j\leq n} \big( \lambda a_i+\mu a_j\big)  $$
\end{example}

The generalized Lambda-determinant reduces to the original Robbins-Rumsey Lambda-determinant
for the choice of parameters $\lambda_a=\lambda$ and $\mu_a=1$ for all $a$.

The generalized Lambda-determinant satisfies a number of properties, inherited
from the symmetries of the $T$-system. We have the following
\begin{prop}\label{apropos}
Let $\sigma,\tau$ be the following transformations of the matrix $A=(a_{i,j})_{i,j\in I}$,
corresponding respectively to quarter-turn rotation, and diagonal reflection:
\begin{equation*}
\sigma(A)_{i,j}=a_{n+1-j,i} \qquad \qquad \tau(A)_{i,j}=a_{j,i}
\end{equation*}
Then with the following action on parameters:
\begin{equation}\label{lamuchoice}
\left\{ 
\begin{matrix} \sigma(\lambda)_{a}=\mu_{-a} \\ 
\sigma(\mu)_{a}=\lambda_a
\end{matrix} \right. 
\quad {\rm and} \quad  \left\{ \begin{matrix} \tau(\lambda)_{a}=\lambda_{-a}\\ \tau(\mu)_a=\mu_a 
\end{matrix} \right. 
\end{equation}
we have: $\vert \varphi(A)\vert_{\varphi(\lambda),\varphi(\mu)}= \vert A\vert_{\lambda,\mu}$
for $\varphi=\sigma,\tau$.
\end{prop}
\begin{proof}
To compute $\vert \varphi(A)\vert_{\varphi(\lambda),\varphi(\mu)}$, we use the following $T$-system relation 
$$T_{i,j,k+1}^\varphi T_{i,j,k-1}^\varphi=\mu_j^\varphi\, 
T_{i,j+1,k}^\varphi T_{i,j-1,k}^\varphi+\lambda_i^\varphi \,T_{i+1,j,k}^\varphi T_{i-1,j,k}^\varphi$$
with $\lambda_i^\varphi=\varphi(\lambda)_i$ and $\mu_j^\varphi=\varphi(\mu)_j$, together with the
initial data 
$$T_{i,j,0}^\varphi=1 \quad (i+j=n\, {\rm mod}\, 2)\qquad T_{i,j,1}^\varphi=\varphi(A)_{{j-i+n+1\over 2},{i+j+n+1\over 2}}
\quad (i+j=n+1\, {\rm mod}\, 2) $$
We wish to compare the solution $T_{0,0,n}^\varphi$ with that, $T_{0,0,n}$, 
of the $T$-system \eqref{inhomtsys}
subject to the initial conditions \eqref{initdat}.
For $\varphi=\sigma,\tau$, we have respectively:
\begin{eqnarray*}
T^\sigma_{i,j,1}&=&\sigma(A)_{{j-i+n+1\over 2},{i+j+n+1\over 2}}=a_{{-i-j+n+1\over 2},{-i+j+n+1\over 2}}=T_{j,-i,1} \\
T^\tau_{i,j,1}&=&\tau(A)_{{j-i+n+1\over 2},{i+j+n+1\over 2}}=a_{{i+j+n+1\over 2},{-i+j+n+1\over 2}}=T_{-i,j,1}
\end{eqnarray*}
From the obvious symmetries of the $T$-system, and with the respective choice of parameters
\eqref{lamuchoice}, we deduce that
$$T_{i,j,k}^\sigma=T_{j,-i,k} \qquad {\rm and} \qquad T_{i,j,k}^\tau=T_{-i,j,k} $$
for all $i,j,k$ such that $i+j+k=n$ mod 2. The proposition follows for $i=j=0$ and $k=n$.
\end{proof}

\subsection{A non-homogeneous example}
\label{secq}

In this section we study a non-trivial example of an inhomogeneous generalized Lambda-determinant
that may be of interest as a statistical model. We choose the matrix $A=(a_{i,j})_{i,j\in I}$ with entries
\begin{equation}\label{one}a_{i,j}=1\qquad (i,j\in I)\end{equation}
We pick coefficients with an explicit dependence on their index, namely
\begin{equation}\label{lamuq}\lambda_a =q^a\qquad \mu_b=q^b  \qquad (a,b\in \Z)\end{equation}
for some fixed parameter $q\in \C^*$. We have the following

\begin{thm}\label{inladet}
The generalized Lambda-determinant of the matrix $A$ \eqref{one} and with the parameters $\lambda,\mu$ of
\eqref{lamuq} reads:
$$ \vert A\vert_{\lambda,\mu}=\prod_{m=1}^{\lfloor \frac{n}{2} \rfloor} \prod_{j=2m-n}^{n-2m} (1+q^j) $$
\end{thm}

To prove the theorem, let us actually solve the more general $T$-system \eqref{inhomtsys}
with coefficients $\lambda_a,\mu_b$ as in \eqref{lamuq}, and with initial conditions \eqref{initdat}.
We have:

\begin{lemma}
The solution of the inhomogeneous $T$-system \eqref{inhomtsys} with initial data $T_{i,j,0}=T_{i,j,1}=1$
for all $i,j\in \Z$ reads:
\begin{eqnarray}
T_{i,j,k}&=&q^{\frac{k(k-1)}{2}{\rm Min}(i,j)}\, \prod_{m=1}^{\lfloor \frac{k-|i-j|}{2} \rfloor}\, \,  
\prod_{a=2m-k+|i-j|}^{k-|i-j|-2m} (1+q^a)\nonumber \\
&&\qquad  \qquad \ \ \times \prod_{m=1}^{|i-j|} \, \, \prod_{a=m}^{k-|i-j|+2m-2}(1+q^a)\label{soltij}
\end{eqnarray}
\end{lemma}
\begin{proof}
By uniqueness of the solution for the given initial data, we simply have to check that the above
satisfies both the initial condition $T_{i,j,0}=T_{i,j,1}=1$ and the $T$-system with coefficients. 
The condition $T_{i,j,0}=T_{i,j,1}=1$ is clear from the formula \eqref{soltij}.
Next, we note that the expression \eqref{soltij} satisfies:
$$\frac{q^j T_{i,j+1,k}T_{i,j-1,k}}{q^iT_{i+1,j,k}T_{i-1,j,k}}=q^{j-i} \qquad (i,j\in\Z;k\in\Z_+) $$
Analogously, we compute:
$$\frac{T_{i,j,k+1}T_{i,j,k-1}}{q^iT_{i+1,j,k}T_{i-1,j,k}}= 1+q^{j-i} \qquad (i,j\in\Z;k\in\Z_{>0})$$
and therefore \eqref{inhomtsys} follows, the lemma is proved.
\end{proof}

Theorem \ref{inladet} follows by taking $i=j=0$, $k=n$ in the expression \eqref{soltij} above.
Note that the result of Theorem \ref{inladet} reduces in the homogeneous limit $q\to 1$ to:
$$ T_{0,0,n}= 2^{\frac{n(n-1)}{2}}$$
in agreement with the result of Example  \ref{maxvan} for $a_i=1$ for all $i$ and $\lambda=\mu=1$.

\section{Computing the generalized Lambda-determinant}

We first write a solution of the $T$-system with inhomogeneous coefficients, based
on a matrix representation generalizing the solution of \cite{DFK12}.

\subsection{Inhomogeneous $T$-system solution I: $V$ and $U$ matrices, definition and properties}

Let us consider the square lattice with vertex set $\Z^2$, and its elementary triangulations, obtained by picking
either of the two possible diagonal edges in each square face. We consider a pair of such triangles
sharing an horizontal edge of the lattice as a generalized rhombus, and we restricted ourselves
to bicolored triangulations such that exactly one of the two triangles
in each generalized rhombus is grey, the other one being white. 
We moreover attach variables to the vertices of the lattice
This gives rise to eight possible generalized rhombi.
In analogy with the solution of the $A_r$ $T$-system \cite{DF}, we associate the following $2\times 2$ matrices
to each of these generalized rhombi, with entries Laurent monomials of the variables at the three vertices
adjacent to the grey triangle, and also depending on fixed coefficients $\lambda,\mu$:
\begin{eqnarray*}
V(d,a,b;\lambda,\mu)&=&\begin{pmatrix} \mu  \frac{a}{b} & \lambda \frac{d}{b} \\ 0 & 1  \end{pmatrix}= 
\raisebox{-1.2cm}{\hbox{\epsfxsize=5.cm \epsfbox{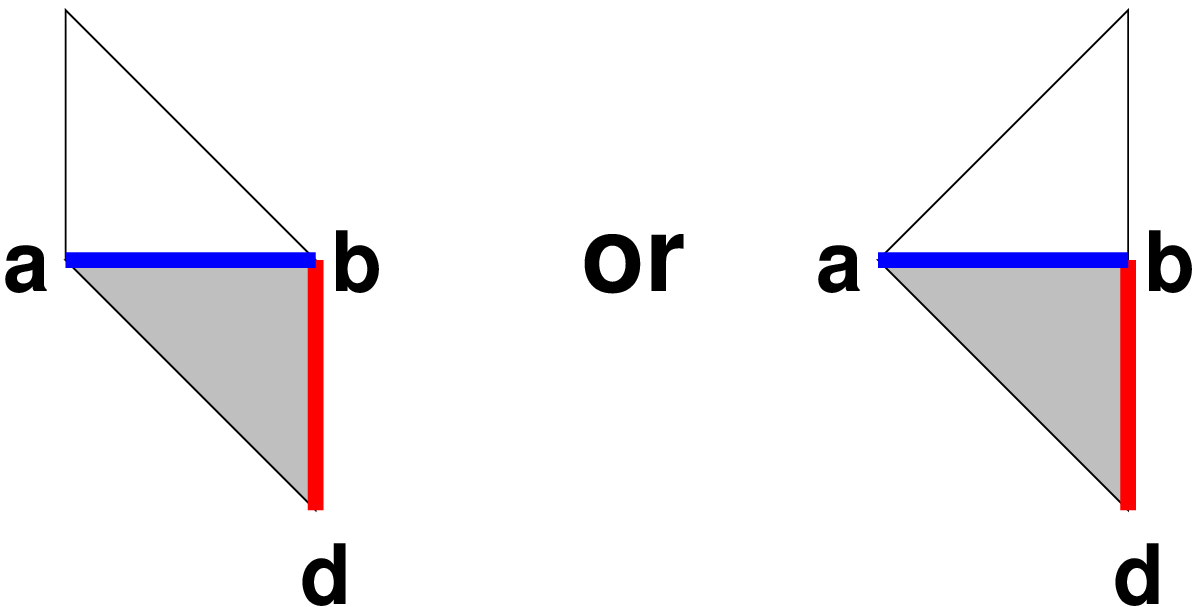}}}\\
V'(d,a,b;\lambda,\mu)&=& \begin{pmatrix}  \frac{a}{b} & \lambda \frac{d}{b} \\ 0 & 1  \end{pmatrix}= 
\raisebox{-1.2cm}{\hbox{\epsfxsize=5.cm \epsfbox{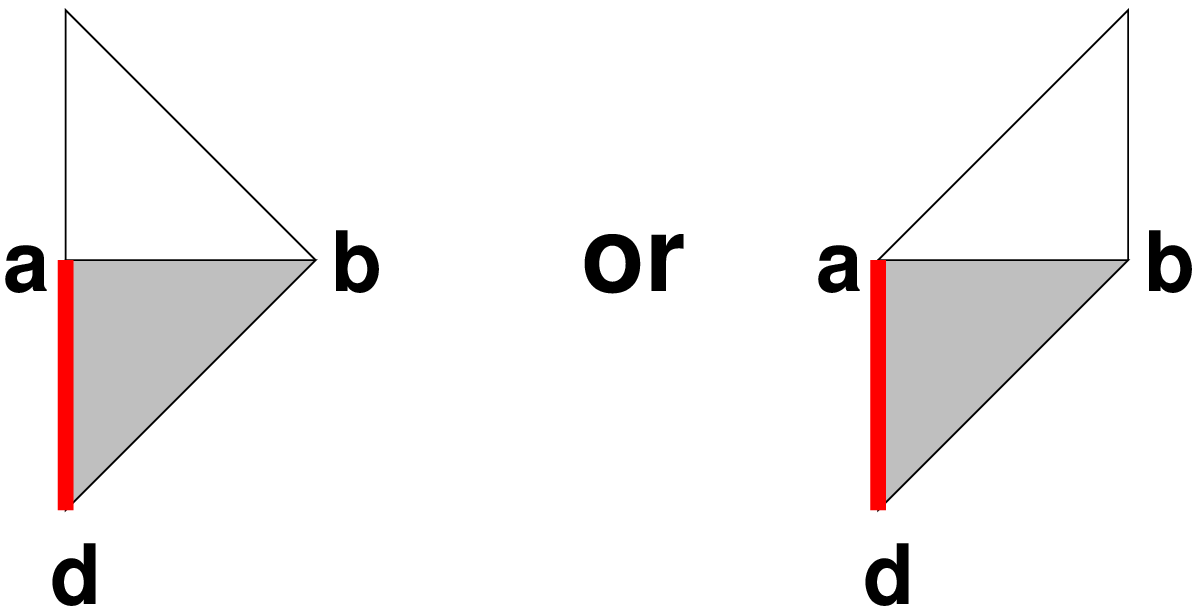}}}\\
U(a,b,c;\lambda,\mu)&=&\begin{pmatrix} 1 & 0 \\  \frac{c}{b} &  \frac{a}{b} \end{pmatrix}= 
\raisebox{-1.cm}{\hbox{\epsfxsize=5.cm \epsfbox{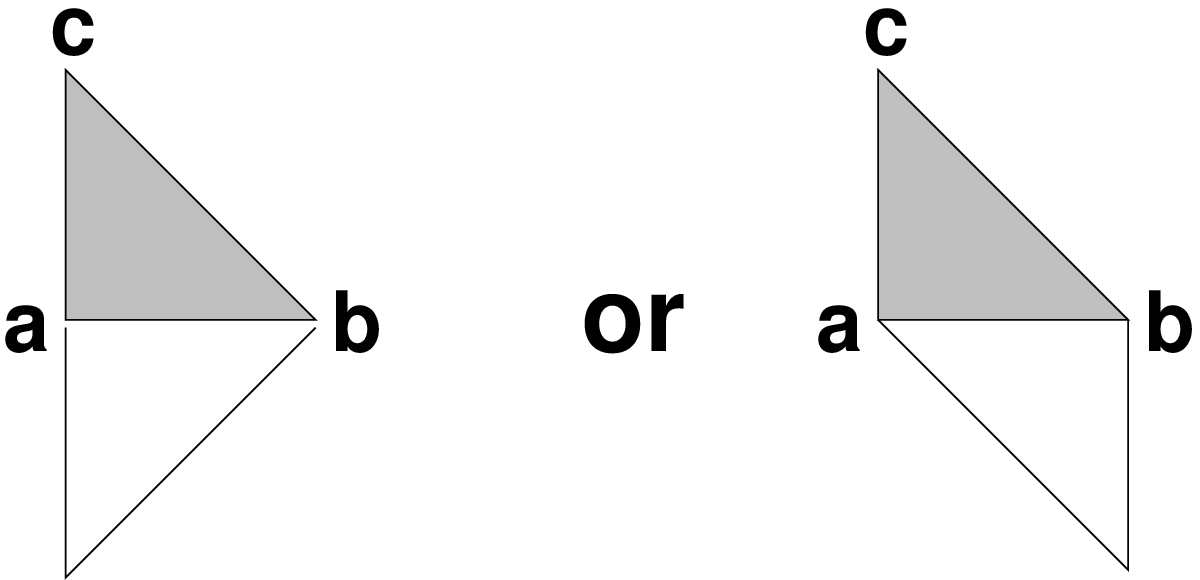}}}\\
U'(a,b,c;\lambda,\mu)&=&\begin{pmatrix} 1 & 0 \\ \frac{c}{b} & \mu \frac{a}{b} \end{pmatrix}= 
\raisebox{-1.cm}{\hbox{\epsfxsize=5.cm \epsfbox{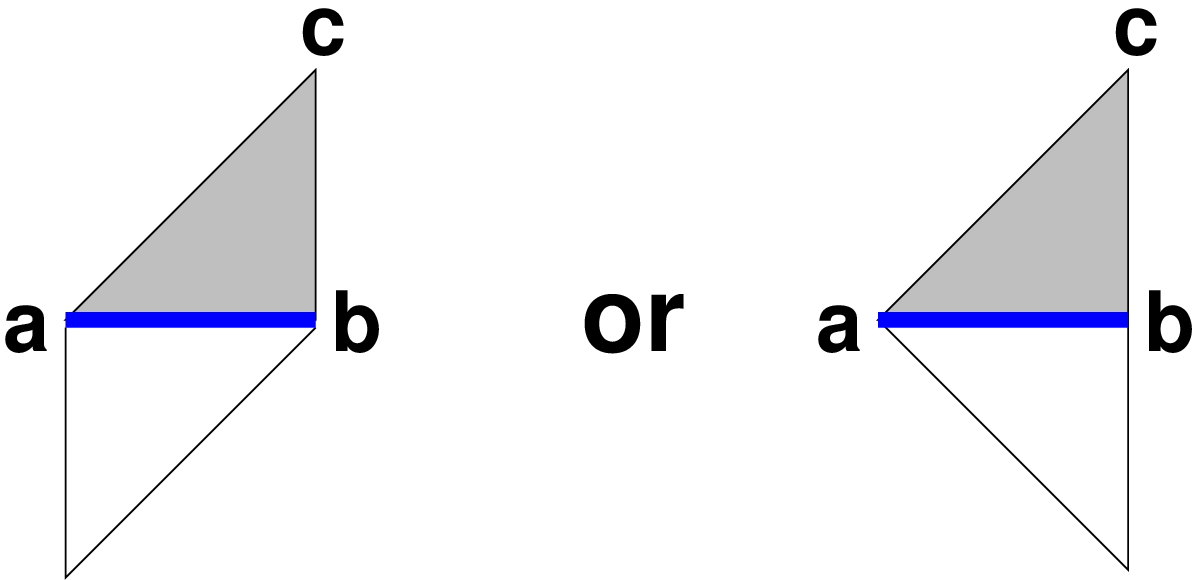}}}
\end{eqnarray*}
Note that we have represented a thicker blue horizontal (resp. red vertical) edge to indicate the presence of the
parameter $\mu$ (resp. $\lambda$) in the corresponding matrices.
Next, we introduce a graphical calculus, by associating to pictures some matrix products.
We read pictures from left to right, and take the product of matrices in the same order.

With this rule, the above  matrices satisfy the following property, easily checked by direct calculation.

\begin{lemma}\label{vu=upvp}
\begin{eqnarray*}\raisebox{-1.5cm}{\hbox{\epsfxsize=3.cm \epsfbox{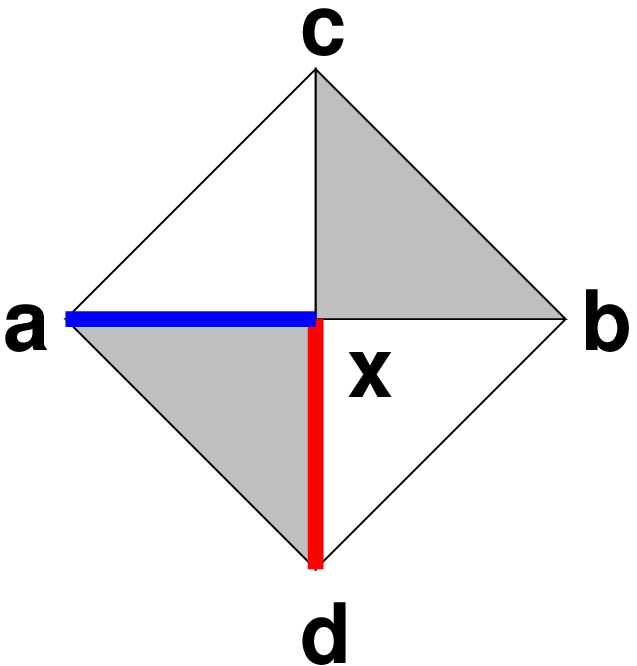}}}&=&V(d,a,x)U(x,b,c)
= U'(a,x',c)V'(d,x',b)=\raisebox{-1.5cm}{\hbox{\epsfxsize=3.cm \epsfbox{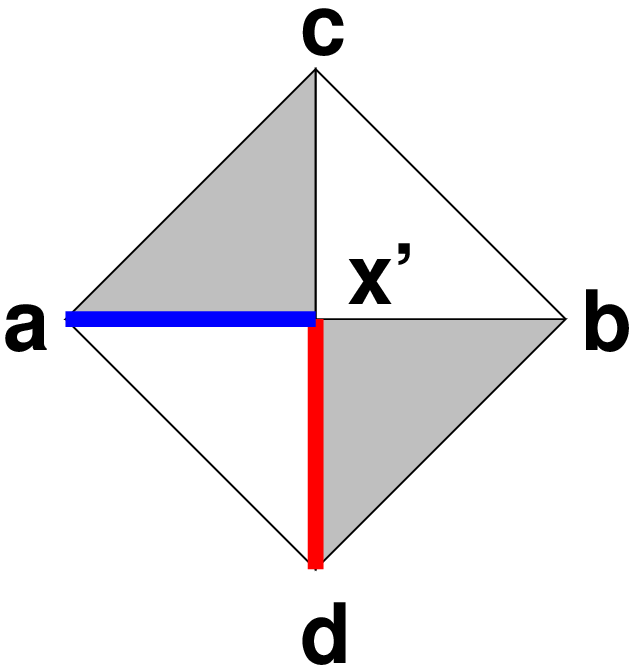}}} \\
&&\qquad {\rm iff}\qquad x x'=\mu a b +\lambda c d
\end{eqnarray*}
\end{lemma}

We use the standard embeddings of any $2\times 2$ matrix $M$ into $SL_{n}$ denoted $M_i$, $i=1,2,...,n-1$
with entries 
\begin{equation*}
(M_i)_{r,s}=\left\{ \begin{matrix} \delta_{r,s} 
& {\rm if}\, r \not\in\{i,i+1\}\, {\rm or}\, s\not\in\{i,i+1\} \\
M_{r,s} & {\rm otherwise} \end{matrix} \right.
\end{equation*}
In this embedding, it is clear that for any $2\times 2$ matrices $M,P$ $M_i$ and $P_j$ commute unless
$j-i\in \{0,\pm 1\}$.
We have also the following exchange properties, easily checked directly:

\begin{lemma}\label{diago}
The matrices $U_i,V_j$ satisfy:
\begin{eqnarray*}
U_{i}(a,b,c;\lambda,\mu)V_{i+1}(b,c,d;\lambda,\mu)
&=&V_{i+1}'(a,c,d;\lambda,\mu)U_i'(a,b,d;\lambda,\mu)\\
\raisebox{-1.8cm}{\hbox{\epsfxsize=2.cm \epsfbox{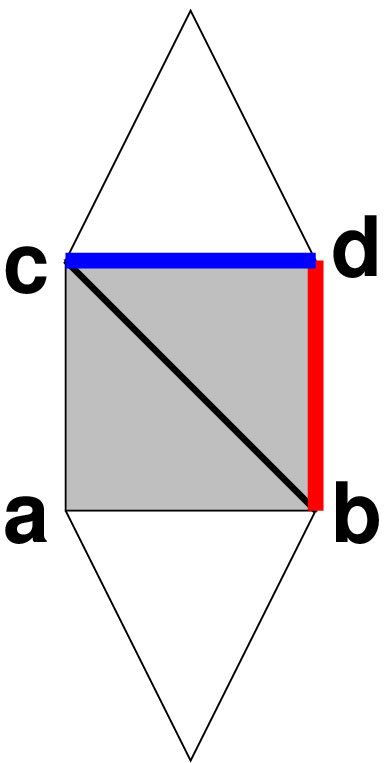}}}
\quad &=&\quad \raisebox{-1.8cm}{\hbox{\epsfxsize=2.cm \epsfbox{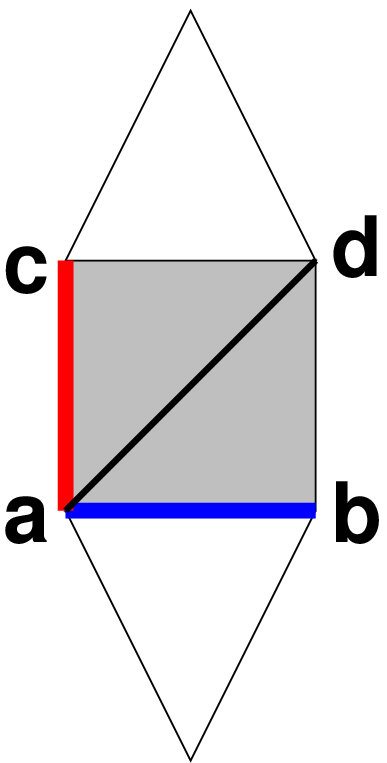}}}
\end{eqnarray*}
where again the pictures are read from left to right and the direction of the diagonal of a square indicates
which generalized rhombus is to the left of the other, and we have represented
the 4 up/down pointing white triangles in-between the two possible configurations of half-squares
each of them may be in (total of 16 configurations here).
We also have the commutation relations:
\begin{eqnarray*}
[V_i(c,a,b),U_j(d,e,f)]&=&[V_i'(c,a,b),U_j(d,e,f)]=[V_i(c,a,b),U_j'(d,e,f)]\\
&=&[V_i'(c,a,b),U_j'(d,e,f)]=0\qquad (j\neq i,i-1)
\end{eqnarray*}
\end{lemma}

The relations of Lemmas \ref{vu=upvp} and \ref{diago} are all we need to construct 
the solutions of the unrestricted $A_\infty$ $T$-system with coefficients, expressed in terms
of products of $U$ and $V$ matrices. The relation of Lemma \ref{vu=upvp}
allows to implement the $T$-system evolution on the variable in the center $x\to x'$. 
The relations of Lemma \ref{diago} allow for rearranging triangulations so as to be able
to apply the former. In particular, when representing white and grey triangulations, we may omit the
diagonal of each square of a uniform color, as both choices of diagonal lead to the same matrix product.
So we will typically consider tilings of domains of $\Z^2$ by means of grey and white unit squares (2 tiles),
and grey and white triangles equal to half a unit square, cut by one diagonal (4 tiles) such that 
any edge is common to two tiles of different colors. We call  such a tiling a square-triangle 
tiling. 

\subsection{Inhomogeneous $T$-system solution II: a determinant formula}

We consider the tilted square domain $D_k(i,j)\subset \Z^2$ defined by
$$ D_k(j,i)=\{ (x,y)\in \Z^2\, {\rm such}\, {\rm that}\, |x-j|+|y-i|\leq k-1\} \qquad (i,j\in \Z;k\in \Z_+)$$
We also consider two square-triangle tilings of $D_k(j,i)$, denoted by $\theta_{min}(k)$ and $\theta_{max}(k)$
defined as follows. $\theta_{min}(k)$ is the intersection of a checkerboard of alternating grey and white squares
with $D_k(j,i)$ such that the NW border of $D_k(j,i)$ is adjacent only to white triangles
(likewise, the SE border is adjacent to only white triangles, while the SW and NE borders are adjacent
only to black triangles). $\theta_{max}(k)$is the only square-triangle tiling without squares having opposite color boundary assignments (NW,SE grey, SW,NE white).
Here are examples of $\theta_{min}(k)$ and $\theta_{max}(k)$ for $k=5$:
$$ \theta_{min}(5)=\raisebox{-2.cm}{\hbox{\epsfxsize=5.cm \epsfbox{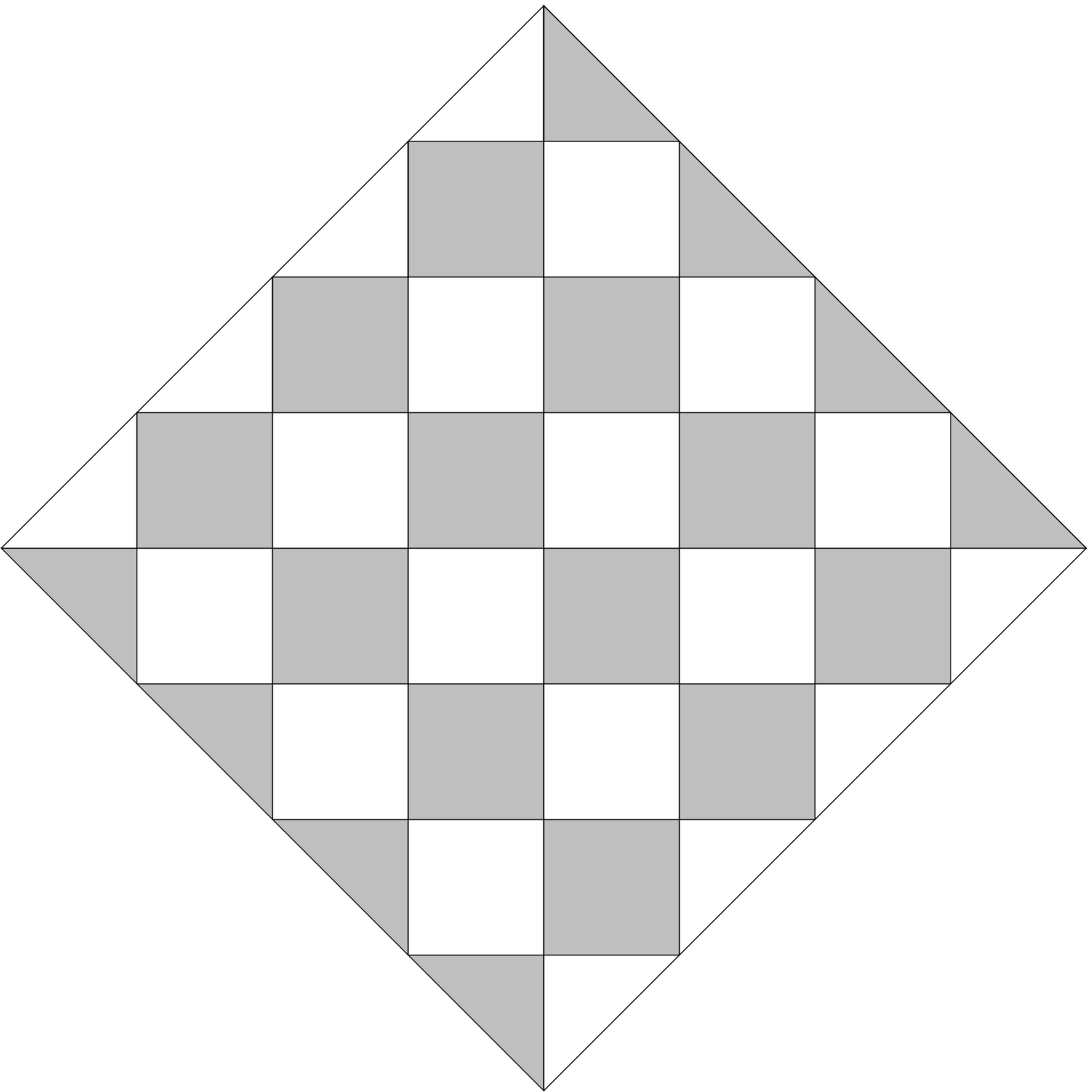}}}\ ,\qquad
\theta_{max}(5)=\raisebox{-2.cm}{\hbox{\epsfxsize=5.cm \epsfbox{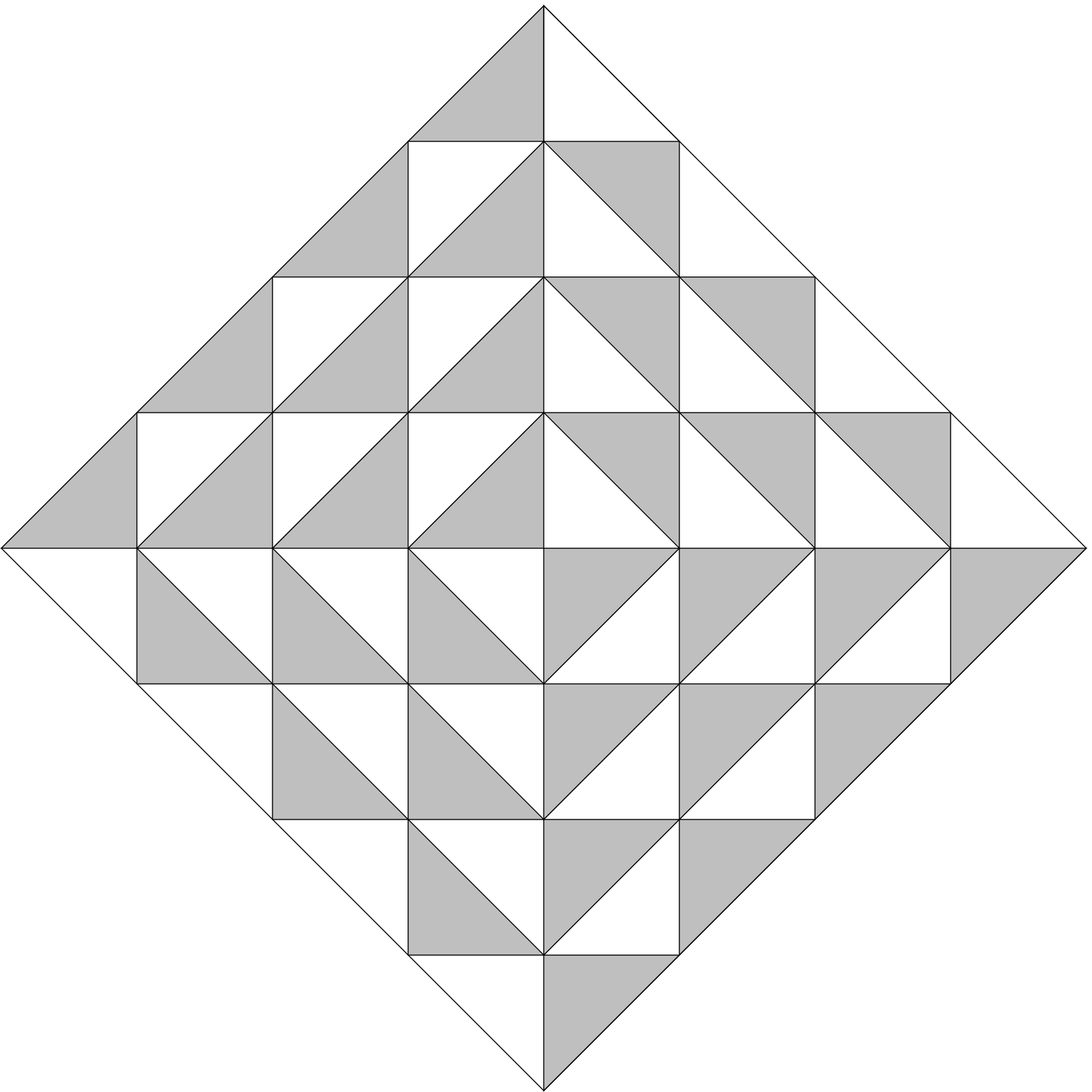}}}$$

To attach matrices to these tilings, let us also include thick (red or blue) edges like in the representations
of $U,V,U',V'$. We fix the ambiguity on grey squares of $\theta_{min}(k)$ (from Lemma \ref{diago}) by 
picking the l.h.s. representation. This gives two decorated tilings ${\tilde\theta}_{min}(k)$ and ${\tilde\theta}_{max}(k)$
which for $k=5$ read:
$${\tilde\theta}_{min}(5)=\raisebox{-2.cm}{\hbox{\epsfxsize=5.cm \epsfbox{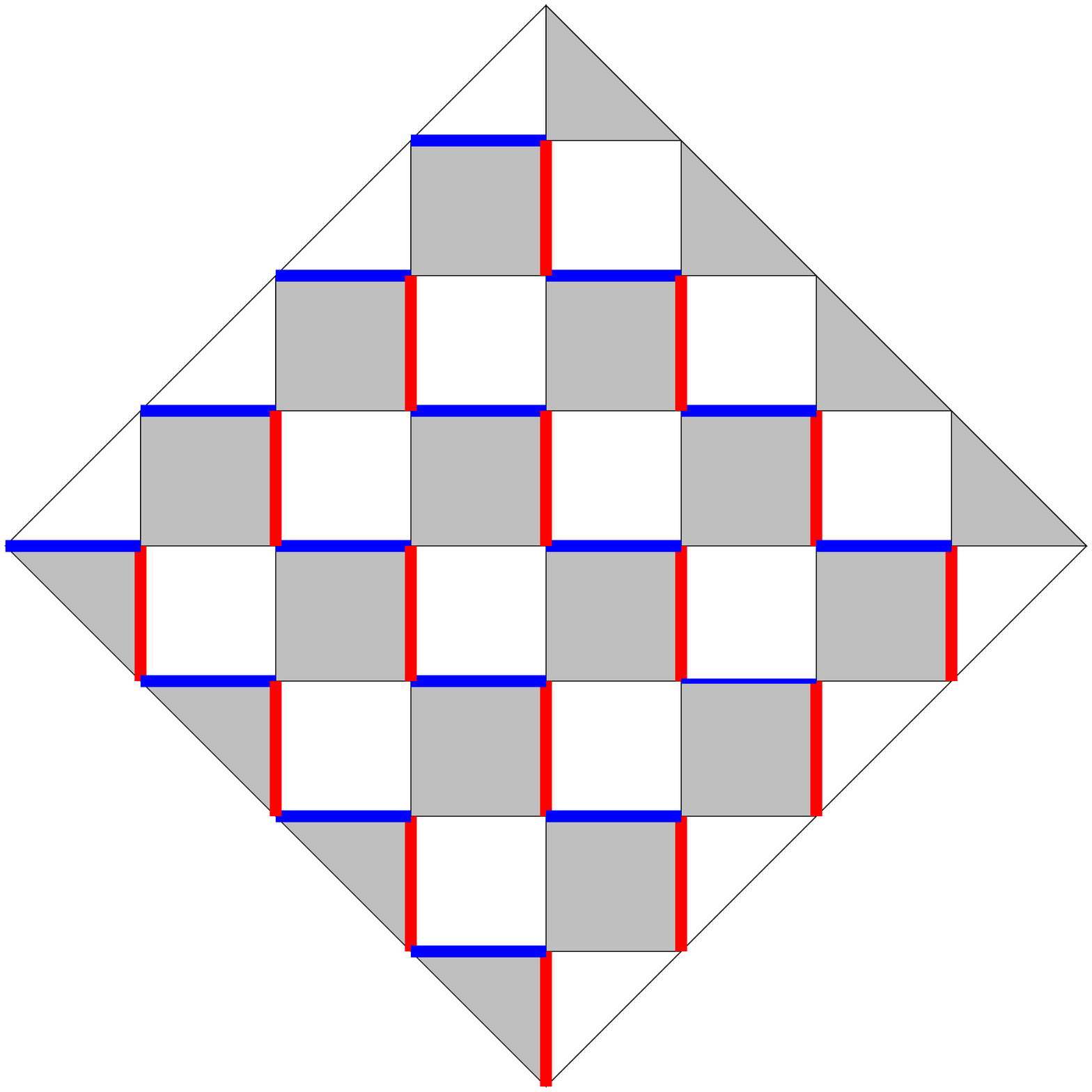}}}\ ,\qquad
{\tilde\theta}_{max}(5)=\raisebox{-2.cm}{\hbox{\epsfxsize=5.cm \epsfbox{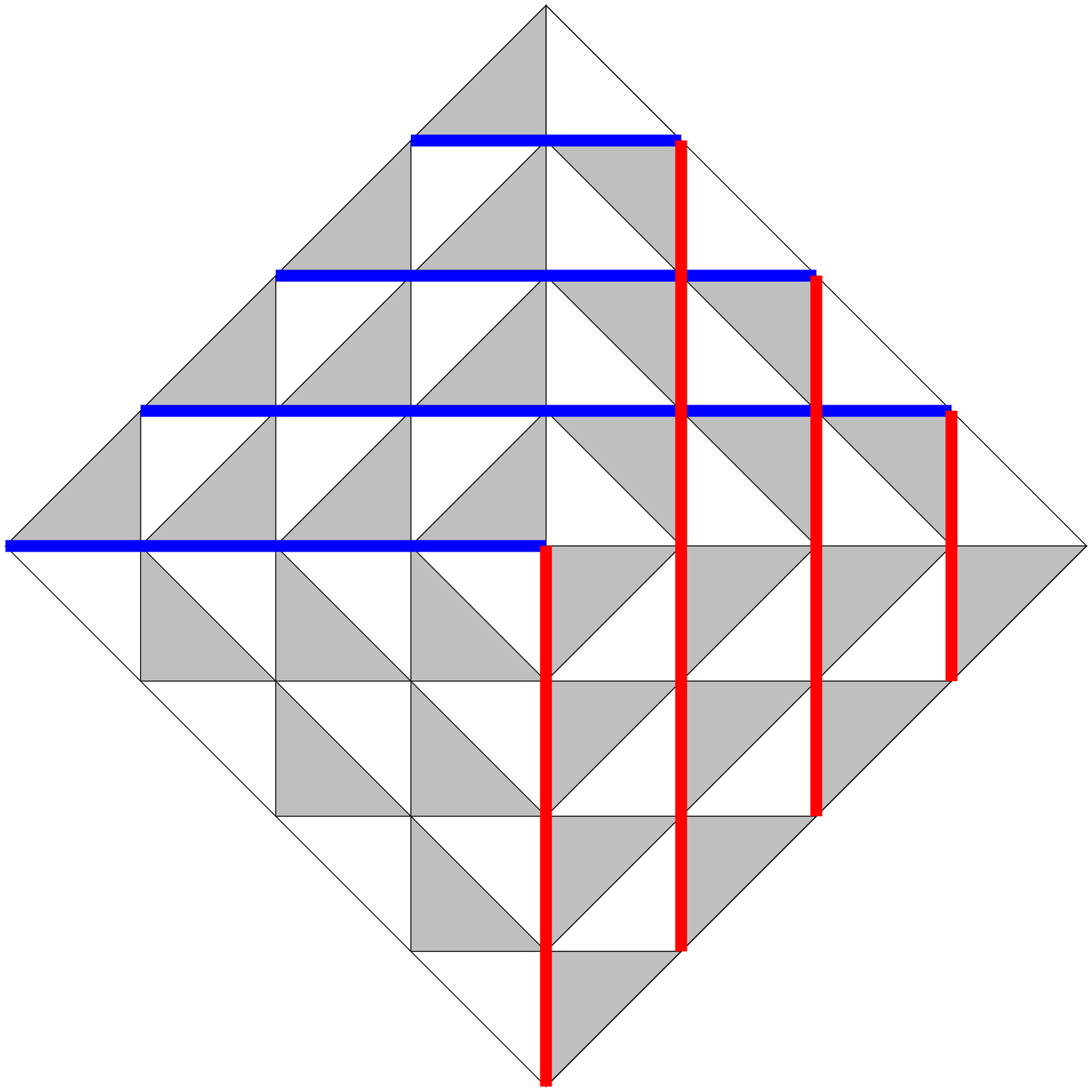}}}$$
We finally attach variables to the vertices and coefficients to the thickened edges of the tilings as follows.
On ${\tilde \theta}_{min}(k)$ we assign variables $t_{i,j}$ to the vertices $(j,i)\in D_k(0,0)$, 
and coefficients $\lambda_i$ to thickened vertical (red) edges with top vertex of the form $(\ell,i)$,
$\mu_j$ to thickened horizontal (blue) edges with right vertex of the form $(j,m)$, $\ell,m\in \Z$.
On ${\tilde \theta}_{max}(k)$ we assign variables $u_{i,j}$ to the vertices $(j,i)\in D_k(0,0)$, 
and coefficients $\lambda_i$ to thickened vertical (red) edges with top vertex of the form $(\ell,i)$,
$\mu_j$ to thickened horizontal (blue) edges with right vertex of the form $(j,m)$, $\ell,m\in \Z$.
For completeness, we assign the value $\mu=1$ (resp. $\lambda=1$) to coefficients corresponding to non-thickened
horizontal (resp. vertical) edges.

We finally attach matrices $\Theta_{min}(k)$ and $\Theta_{max}(k)$ to the two tilings.
These are defined as the products of $U_i,V_i,U'_i,V'_i$ matrices embedded into $SL_{2n-2}$,
coded respectively by the two decorated tilings with assigned vertex and thickened edge values.

\begin{example}
For $k=3$, we have:
\begin{eqnarray*} \Theta_{min}(3)&=& 
V_2(t_{-1,-1},t_{0,-2},t_{0,-1};\lambda_{0},\mu_{-1})V_1(t_{-2,0},t_{-1,-1},t_{-1,0};\lambda_{-1},\mu_0)\\
&&\times U_2(t_{0,-1},t_{0,0},t_{1,-1};1,1)V_3(t_{0,0},t_{1,-1},t_{1,0};\lambda_1,\mu_0)\\
&&\times U_1(t_{-1,0},t_{-1,1},t_{0,0};1,1)V_2(t_{-1,1},t_{0,0},t_{0,1};\lambda_0,\mu_1)\\
&&\times U_3(t_{1,0},t_{1,1},t_{2,0};1,1)U_2(t_{0,1},t_{0,2},t_{1,1};1,1) 
\end{eqnarray*}
corresponding to the following variables and coefficients:
$$ \raisebox{-3.cm}{\hbox{\epsfxsize=8.cm \epsfbox{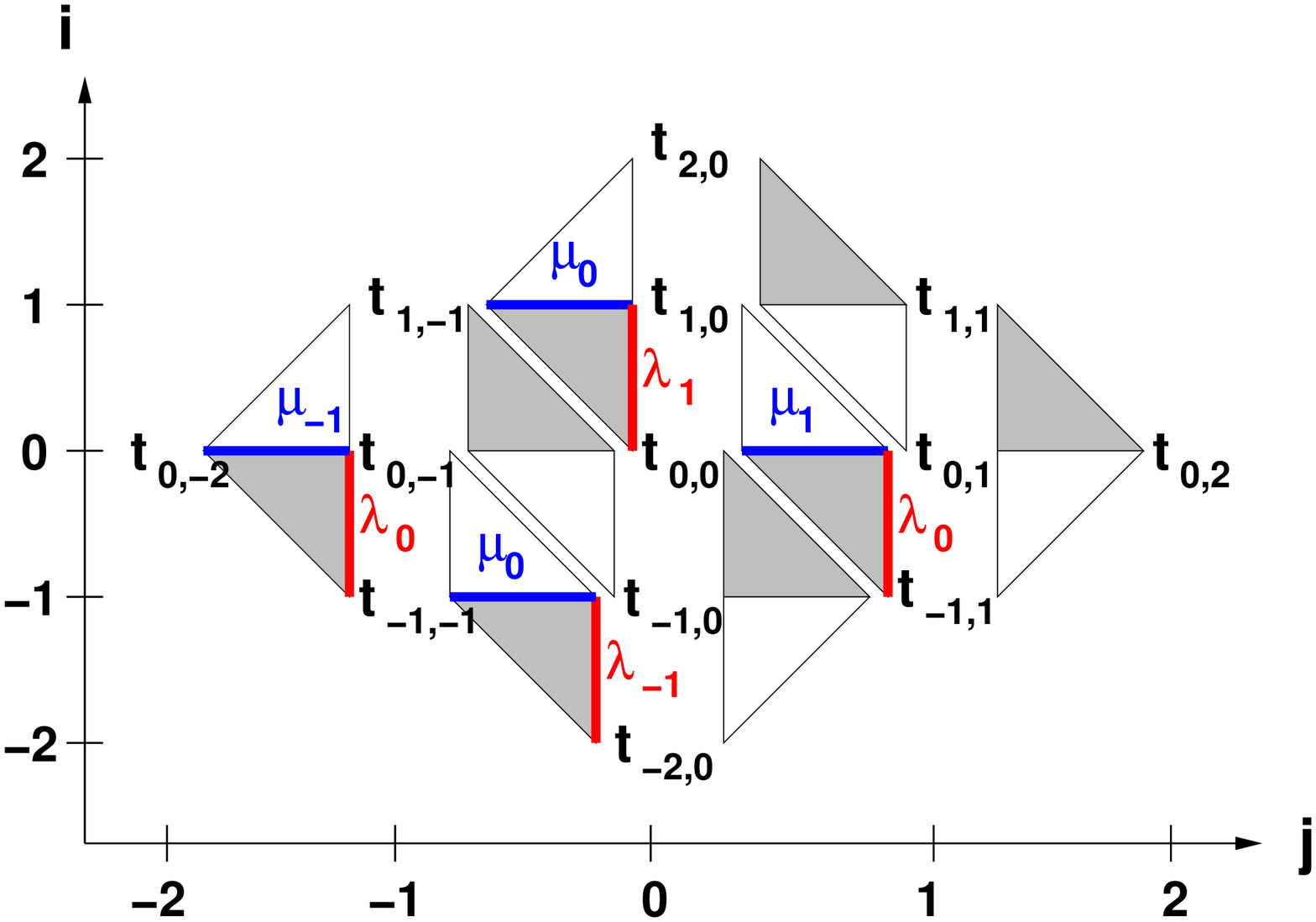}}} $$
Likewise, we have:
\begin{eqnarray*} \Theta_{max}(3)&=& 
U_2'(u_{0,-2},u_{0,-1},u_{1,-1};1,\mu_{-1})U_1(u_{-1,-1},u_{-1,0},u_{0,-1};1,1)\\
&&\times U_3'(u_{1,-1},u_{1,0},u_{2,0};1,\mu_0)U_2'(u_{0,-1},u_{0,0},u_{1,0};1,\mu_0)\\
&&\times V_2'(u_{-1,0},u_{0,0},u_{0,1};\lambda_0,1)V_1'(u_{-2,0},u_{-1,0},u_{-1,1};\lambda_{-1},1)\\
&&\times V_3(u_{0,1}u_{1,0},u_{1,1};\lambda_1,\mu_1)V_2'(u_{-1,1},u_{0,1},u_{0,2};\lambda_0,1) 
\end{eqnarray*}
corresponding to the following variables and coefficients:
$$ \raisebox{-3.cm}{\hbox{\epsfxsize=8.cm \epsfbox{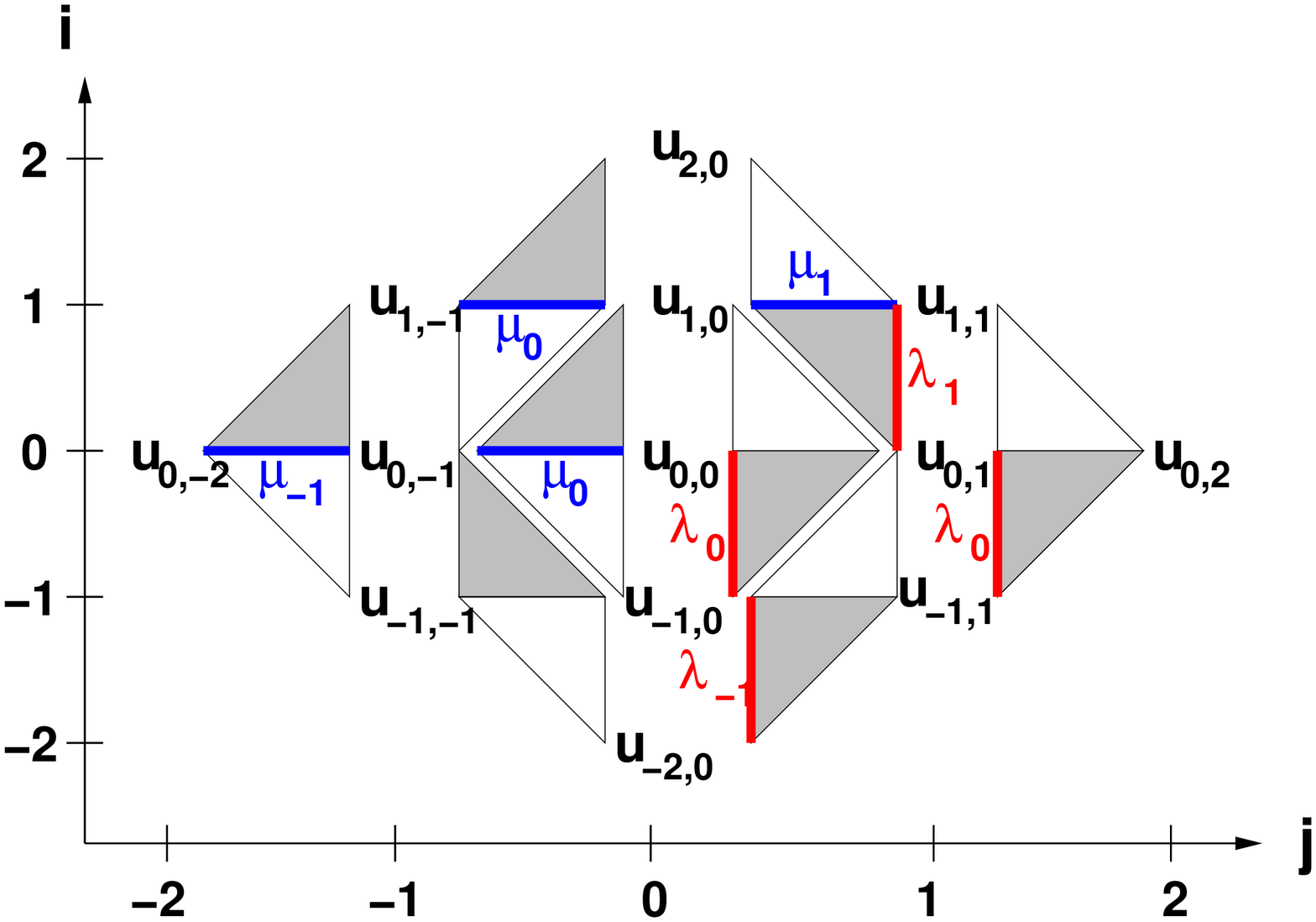}}} $$
\end{example}

We have:

\begin{thm}\label{identet}
Let $n\geq 1$, and let us pick variables $u_{i,j}$ as:
\begin{equation}\label{defu} u_{i,j}=T_{i,j,n-|i|-|j|} \qquad (j,i)\in D_n(0,0) 
\end{equation}
where $T_{i,j,k}$, for $i+j+k=n$ mod 2, is the solution of the $T$-system with coefficients \eqref{inhomtsys}
subject to the initial conditions 
\begin{equation}\label{tinit} T_{i,j,\epsilon_{i,j,n}}=t_{i,j}\qquad (j,i)\in D_n(0,0) 
\end{equation} 
where $\epsilon_{i,j,n}=(i+j+n\, {\rm mod}\, 2)\in \{0,1\}$.
Then we have the matrix identity
$$ \Theta_{min}(n)=\Theta_{max}(n) $$
\end{thm}
\begin{proof}
The identity is proved by showing that we may transform the minimal tiling ${\tilde \theta}_{min}(n)$
with variables $t_{i,j}$ into the maximal tiling $\theta_{max}(n)$ with variables $u_{i,j}$ by successive
applications of the Lemmas \ref{vu=upvp} and \ref{diago}. This is illustrated below for the case $n=3$.
$$  \raisebox{-2.cm}{\hbox{\epsfxsize=16.cm \epsfbox{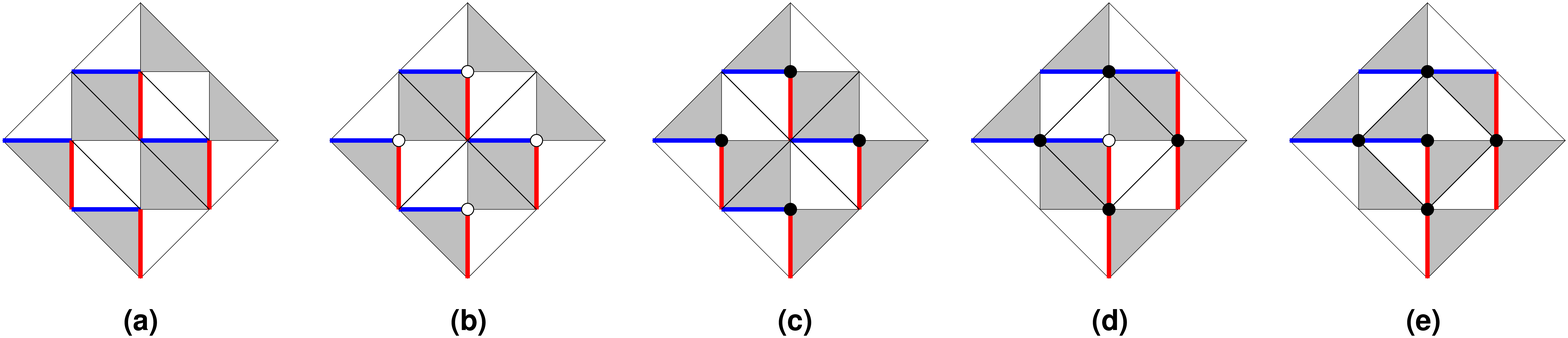}}} $$
Starting from 
${\tilde \theta}_{min}(n)$ (a), let us first apply Lemma \ref{diago} to flip the diagonals of all the white squares
(b).
Then apply Lemma \ref{vu=upvp} to update all the values at vertices $(j,i)\in D_n(0,0)$ such that
$i+j=n$ mod 2 (marked as empty circles on (b)). 
Due to the update relation $xx'=\mu ab+\lambda cd$, and comparing with the
$T$-system evolution \eqref{inhomtsys}, we see that the updated values are $T_{i,j,2}$ (filled vertices on (c)). 
We then use Lemma \ref{diago} to flip the diagonals of all the remaining squares (d), then update 
the values at vertices $(j,i)\in D_{n-1}(0,0)$ such that
$i+j=n-1$ mod 2 (empty circle on (d)), yielding updated values $T_{i,j,3}$. The procedure is iterated until
there are no more squares. In the final tiling configuration, the boundary triangles have opposite colors from
those of ${\tilde \theta}_{min}(n)$, the tiling is therefore ${\tilde \theta}_{max}(n)$.
Note that the last (central) updated value is $u_{0,0}=T_{0,0,n}$. As both Lemmas \ref{vu=upvp} and \ref{diago}
leave the result of the corresponding matrix products unchanged, the theorem follows.
\end{proof}

\begin{cor}\label{identa}
Let $a_{min}(n)$, $a_{max}(n)$ denote respectively the $n-1\times n-1$ principal minors
of $\Theta_{min}(n)$ and $\Theta_{max}(n)$, then we have:
$$a_{min}(n)=a_{max}(n)$$
\end{cor}

Our last task is to relate the central value $T_{0,0,n}=u_{0,0}$ to the matrix $\Theta_{min}(n)$.
We have the following:

\begin{thm}\label{soltsys}
Notations are as in Theorem \ref{identet}.
The solution $T_{0,0,n}$ of the $T$ system with coefficients \eqref{inhomtsys} 
is expressed in terms of its initial data $t_{i,j}$ as:
$$T_{0,0,n}=a_{min}(n) \prod_{i=2-n}^{-1} t_{i,1-n-i}^{-1}\prod_{i=2-n}^{0}t_{i,n-1+i} $$
\end{thm}
\begin{proof}
Let us compute $a_{max}(n)$. 
We note that, as $U_i,U'_i$ are lower triangular and $V_i,V_i'$ are upper triangular, the $2n-2\times 2n-2$ matrix
$\Theta_{max}(n)$ is the product of a lower triangular $\mathcal L$ matrix by an upper one 
$\mathcal U$, corresponding 
respectively to the left and right halves of ${\tilde \theta}_{max}(n)$. More precisely, all matrices in the 
SW corner are of $U_i$ type ($i\leq n-2$), all matrices in the NW corner are of $U'_i$ type ($n-1\leq i\leq 2n-3$),
all matrices in the SE corner are of $V'_i$ type ($i\leq n-1$) and all matrices in the NE corner are of $V_i$ type
($n\leq i\leq 2n-3$).

Let $M_{max}(n)$ be the truncation of $\Theta_{max}(n)$ to its
$n-1$ first rows and columns.
From the triangularity property, we have $M_{max}(n)=\ell u$
where $\ell$ and $u$ are the truncations of $\mathcal L$ and $\mathcal U$ to their $n-1$ first rows and columns,
so that $a_{max}(n)=\det(\ell)\det(u)$. By the triangularity property, and the above remark on
corners, $\det(\ell)$ is the product over all diagonal
matrix elements of the $U_i$ in $\Theta_{max}(n)$, while
$\det(u)$ is the product over all diagonal matrix elements of the $V_i'$ in $\Theta_{max}(n)$.
These in turn have a very simple interpretation in terms of the tilings: the diagonal matrix elements of $U_i$
($i\leq n-2$) are the ratios (left vertex value)/(right vertex value) along the horizontal edges of the tiling (here in the
strict SW corner, i.e. with $2-n\leq i<0$ and $1-n\leq j\leq 0$. The product over these is telescopic
and leaves us with only the ratios of (leftmost vertex value)/(rightmost vertex value)=$u_{i,i+1-n}/u_{i,0}$
along the row $i$.
Analogously, the diagonal matrix elements of $V_i'$ ($i\leq n-1$) have the same interpretation,
and we get the telescopic products $u_{i,0}/u_{i,n-1+i}$ along the row $i$. Collecting both products, we get
$$a_{max}(n)=\det(\ell)\det(u)
=\prod_{i=1-n}^{-1} \frac{u_{i,1-n-i}}{u_{i,0}} \prod_{i=2-n}^{0}\frac{u_{i,0}}{u_{i,n-1+i}}= u_{0,0} 
\prod_{i=2-n}^{-1} u_{i,1-n-i}\prod_{i=1-n}^{0}u_{i,n-1+i}^{-1} $$
Finally, we note that the SW and SE boundary values of $u_{i,j}$ are all of the form $T_{i,j,1}$ and therefore
are equal to the corresponding $t_{i,j}$ variables. The theorem then follows from Corollary \ref{identa},
allowing to substitute $a_{min}(n)$ for $a_{max}(n)$.
\end{proof}

\subsection{A determinant formula for the generalized Lambda-determinant}

To get a formula for the generalized Lambda-determinant, we simply have to combine the results
of Theorems \ref{identet} and \ref{soltsys}. Our generalized Lambda-determinant is the solution
$T_{0,0,n}$ of the $T$-system with coefficients subject to the initial conditions \eqref{initdat}.
It corresponds to the initial values:
\begin{equation}\label{condinit} t_{i,j}=\left\{\begin{matrix}  1 & {\rm if} \, i+j+n=0\, {\rm mod}\, 2\\
a_{\frac{j-i+n+1}{2},\frac{j+i+n+1}{2}} & {\rm otherwise} 
\end{matrix}\right. 
\end{equation}
We assume from now on that the initial data $t_{i,j}$ is chosen as in \eqref{condinit}.
With the same notations as in the previous section, we obtain:

\begin{thm}\label{ladet}
The generalized Lambda-determinant of any matrix $A=(a_{i,j})_{1\leq i,j\leq n}$ reads:
$$ \vert A\vert_{\lambda,\mu} = a_{min}(n) \prod_{i=2}^{n}a_{i,1}^{-1}\prod_{i=1}^{n}a_{n,i}$$
As such, it is a Laurent polynomial of the $a_{i,j}$'s with coefficients in 
$\Z_+[(\lambda_i)_{i\in J},(\mu_i)_{i\in J}]$, where $J=\{1-n,2-n,...,n-1\}$.
\end{thm}
\begin{proof}
By direct application of Theorems \ref{identet} and \ref{soltsys}, with the initial conditions
\eqref{condinit}.
As $a_{min}(n)$ is a principal minor of a product of $U_i,U'_i,V_i,V'_i$ matrices, 
and as the entries of all the matrices are Laurent monomials of the $t_{i,j}$'s
and monomials of the $\lambda_i$'s and $\mu_j$'s, we recover the Laurent polynomiality
property of cluster algebras with coefficients, namely that the solution $T_{0,0,n}$
is a Laurent polynomial of the entries of $A$, with coefficients that are polynomials of the 
$\lambda_i$'s and $\mu_j$'s. Moreover, from the form of the matrices, these polynomial 
coefficients have themselves non-negative integers coefficients. The theorem follows.
\end{proof}

\section{From networks to 6V model}

\subsection{Network formulation}

The matrices $V,V',U,U'$ may be interpreted as elementary ``chips" that may be used to build
networks, i.e. directed graphs with weighted edges. Each non-zero matrix element $m_{i,j}$ of $V,V',U,U'$
is interpreted as a the weight of a directed edge connecting an entry vertex $i$ to an exit vertex $j$. This
gives rise to the four following ``chips":
\begin{eqnarray}
V(d,a,b;\lambda,\mu)&=& 
\raisebox{-1.2cm}{\hbox{\epsfxsize=6.cm \epsfbox{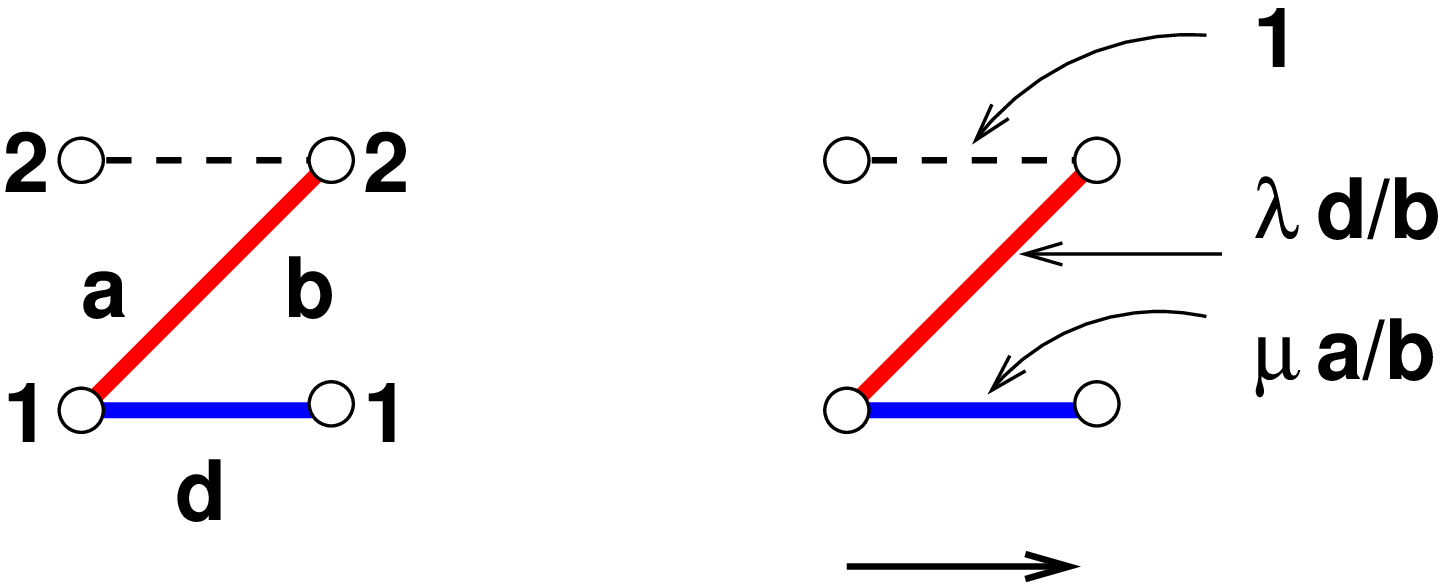}}}\nonumber \\
V'(d,a,b;\lambda,\mu)&=&
\raisebox{-1.2cm}{\hbox{\epsfxsize=6.cm \epsfbox{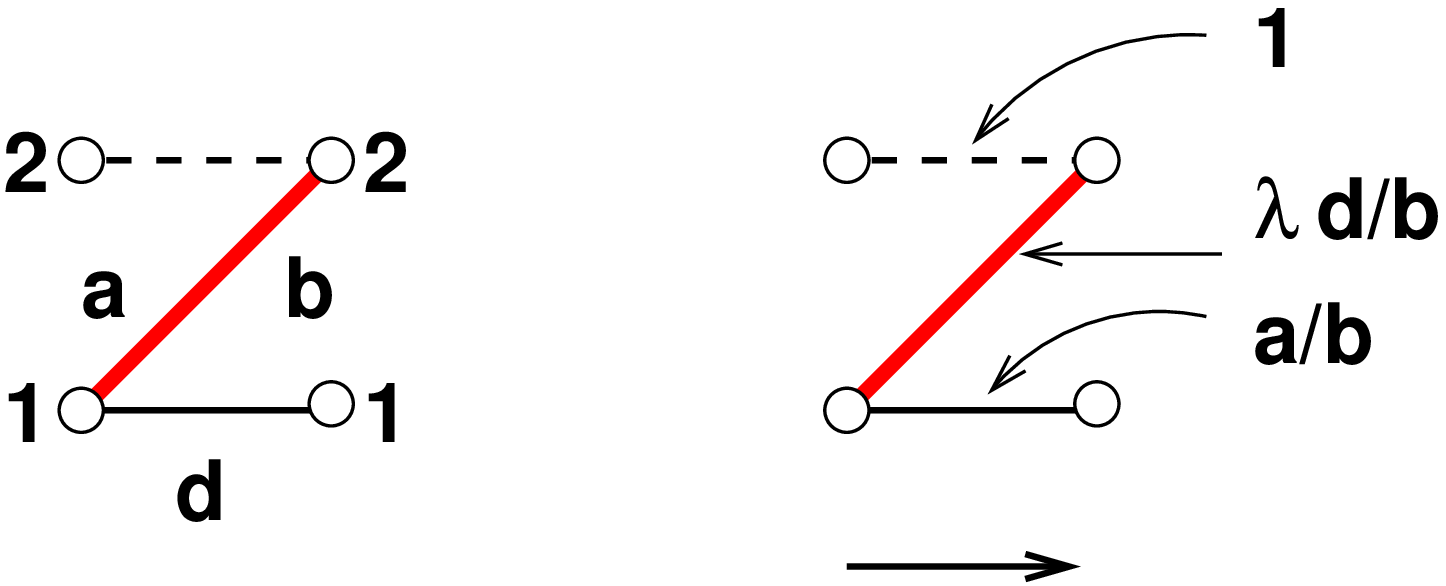}}}\nonumber \\
U(a,b,c;\lambda,\mu)&=&
\raisebox{-1.2cm}{\hbox{\epsfxsize=6.cm \epsfbox{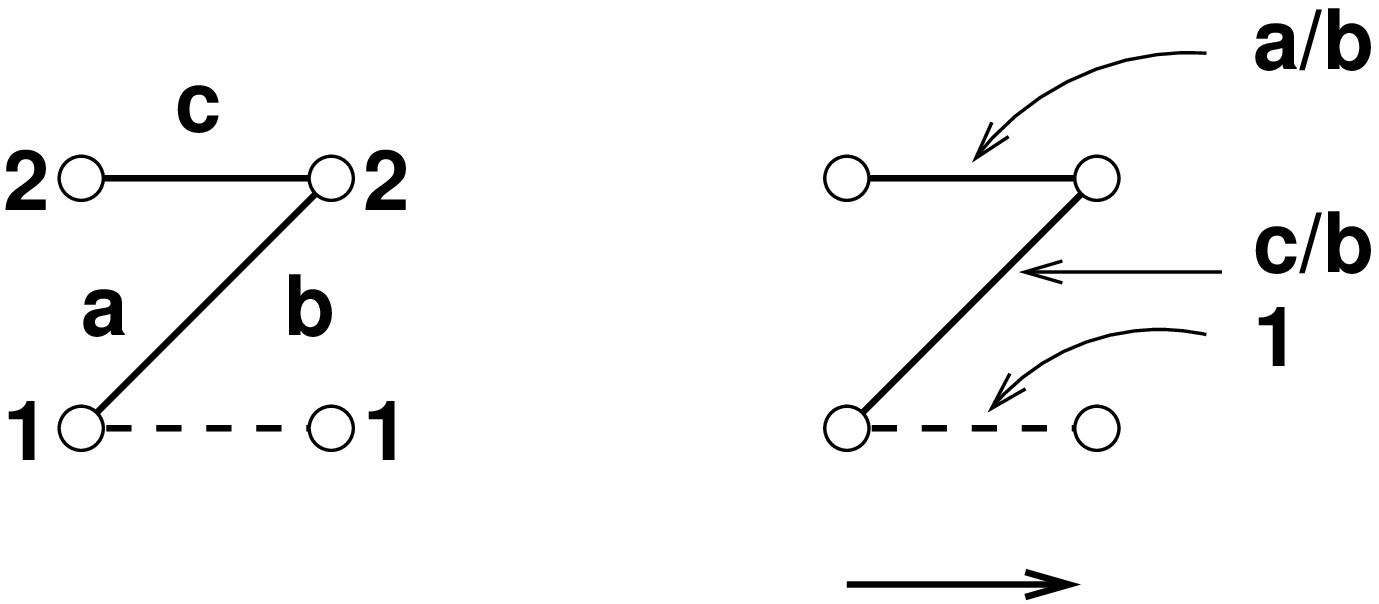}}}\nonumber \\
U'(a,b,c;\lambda,\mu)&=&
\raisebox{-1.2cm}{\hbox{\epsfxsize=6.cm \epsfbox{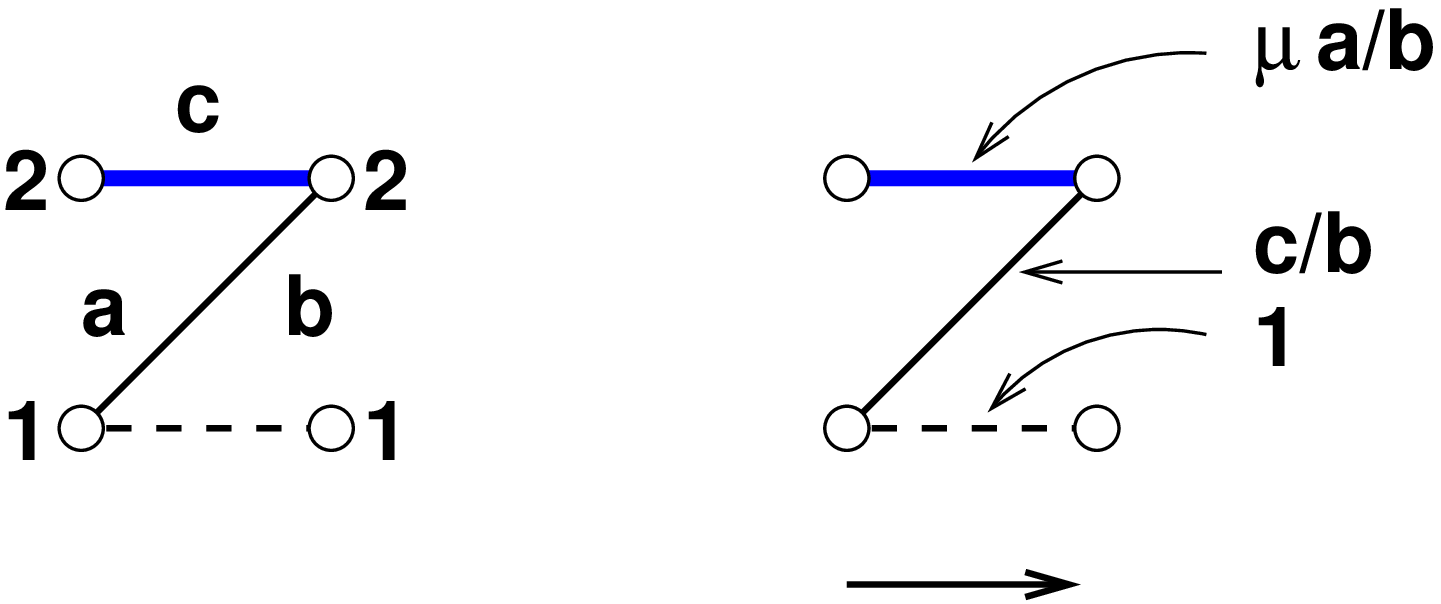}}}\label{chipsrule}
\end{eqnarray}
where we have represented in dashed black lines the edges with weight $1$, and
in thickened red (resp. blue) edges the diagonals (resp. horizontals)
carrying an extra multiplicative weight $\lambda$ (resp. $\mu$).
The $SL_r$ embedding $M_i$ is clear: we place the chip for $M$ in positions $i,i+1$
and complete the graph by horizontal edges $j\to j$ with weight $1$, for $j\neq i,i+1$.
The product of two matrices correspond in this language to the concatenation of the graphs,
namely the identification of the exit points of the left matrix graph with the entry points 
of the right matrix graph. Graphs obtained by concatenation of elementary chips are called networks.
Note that in the present case networks have face labels, that determine all the edge weights via the
rules \eqref{chipsrule}.

The matrix of the network is the product of the matrices of the chips forming it. The matrix element
$(i,j)$ of this product is the sum over all directed paths on the network with entry $i$ and exit $j$ of the
product of the weights of the traversed edges.

We may now represent the networks $G_{min}(n)$ and $G_{max}(n)$ corresponding 
to the matrices $\Theta_{min}(n)$ and $\Theta_{max}(n)$. For $n=5$, we have (recall that dashed black
lines have weight $1$, while thickened blue (resp. red) lines have an extra weight $\mu$ (resp. $\lambda$)):
$$G_{min}(5)=\raisebox{-2.cm}{\hbox{\epsfxsize=12.cm \epsfbox{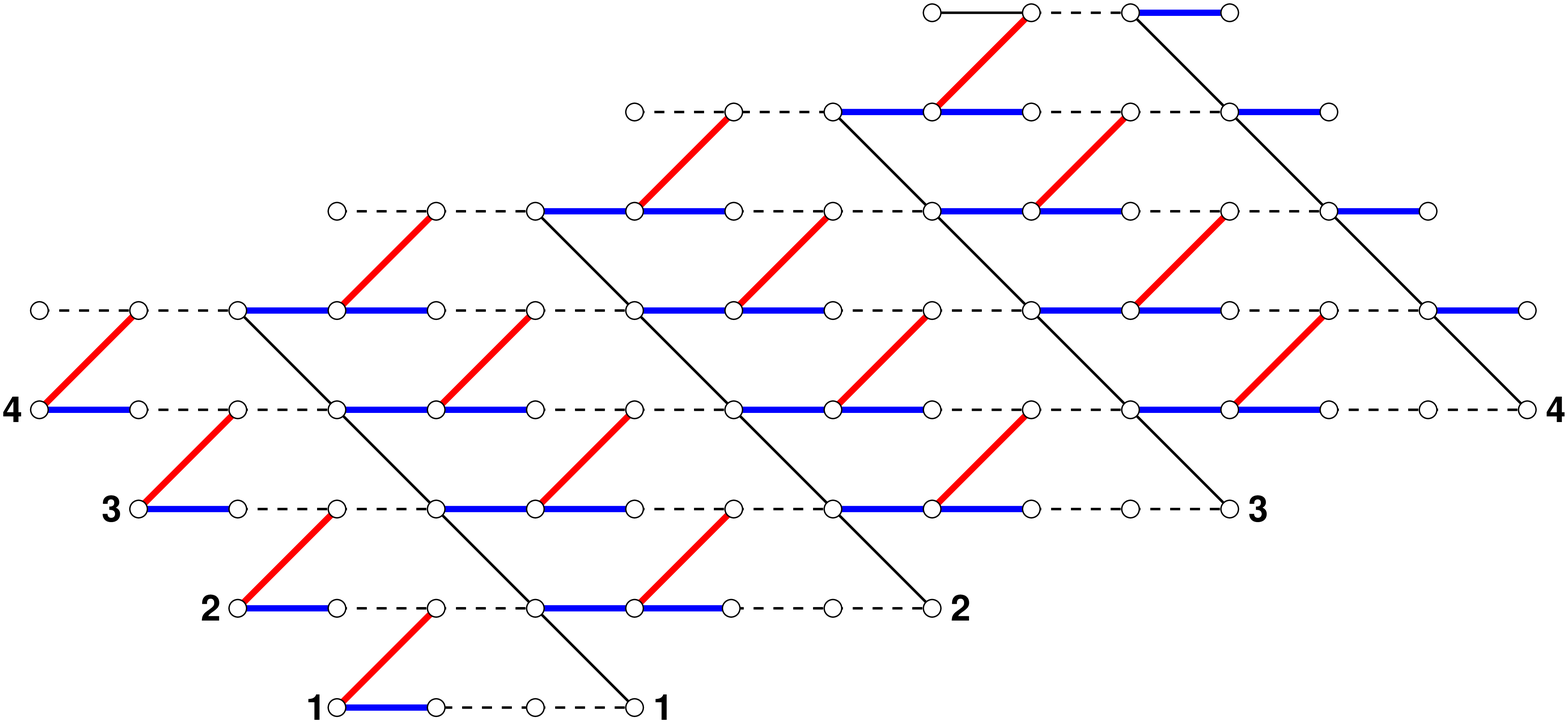}}}$$
and
$$G_{max}(5)=\raisebox{-2.cm}{\hbox{\epsfxsize=12.cm \epsfbox{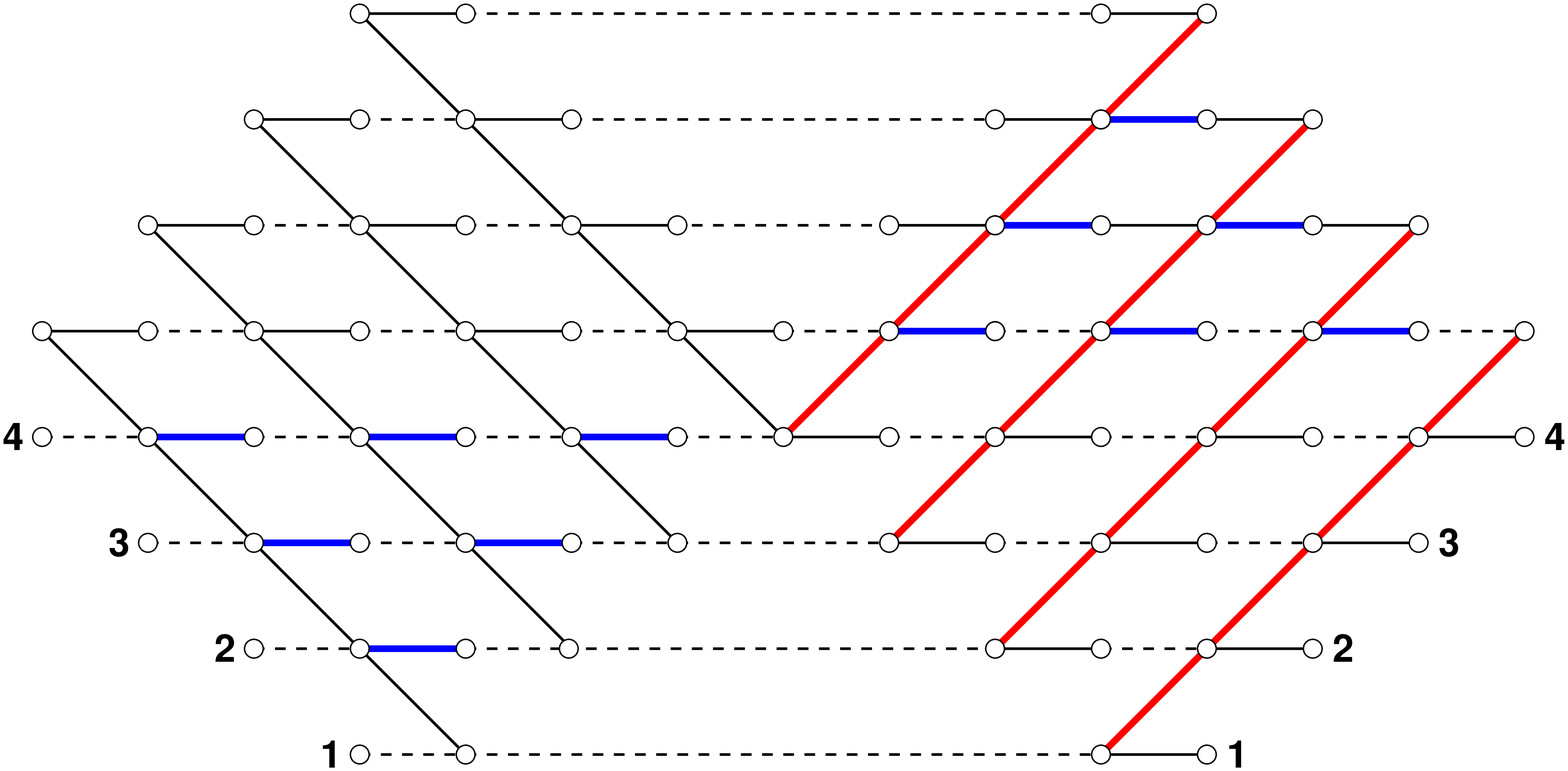}}}$$

We have:

\begin{lemma}\label{aminmax}
The function $a_{min}(n)$ (resp. $a_{max}(n)$) is the partition function for families
of non-intersecting paths that start at the entry points $1,2,...,n-1$ and end at the exit points $1,2,...,n-1$
of the network $G_{min}(n)$ (resp. $G_{max}(n)$).
\end{lemma}
\begin{proof}
By direct application of the Lindstr\"om-Gessel-Viennot Theorem \cite{LGV1,LGV2}, as $a_{min}(n)$ (resp. $a_{max}(n)$)
are the $n-1\times n-1$ principal minors of $\Theta_{min}(n)$ (resp. $\Theta_{max}(n)$), the matrices of the networks
$G_{min}(n)$ (resp. $G_{max}(n)$).
\end{proof}

%By use of the Lindstr\"om-Gessel-Viennot Theorem [xxx], we may now interpret the $n-1\times n-1$ 
%principal minor of $\Theta_{min}(n)$ and $\Theta_{max}(n)$ as the partition function for families
%of non-intersecting paths that start at the entry points $1,2,...,n-1$ and end at the exit points $1,2,...,n-1$
%(these are indicated on the $n=5$ examples above). That is, the quantities $a_{min}(n)$ and $a_{max}(n)$
%are the sums over the non-intersecting paths on the networks $G_{min}(n)$ and $G_{max}(n)$ of the products
%of the traversed edge weights. 

\begin{remark}
Note that due to the structure of $G_{max}(n)$, there is a unique configuration of $n-1$ non-intersecting
paths contributing to $a_{max}(n)$, namely that in which all the paths have only horizontal steps.
This gives an alternative pictorial interpretation of the proof of Theorem \ref{soltsys}.
\end{remark}

In view of Theorem \ref{ladet}, let us now reformulate $a_{min}(n)$. To this effect, we will simplify slightly the
network $G_{min}(n)$ by use of the following:

\begin{lemma}\label{equival}
We have the following equivalence between networks corresponding to the matrix product 
$U_{i-1}(a,b,c;1,1)V_{i}(b,c,d;\lambda,\mu)$:
$$  \raisebox{-1.8cm}{\hbox{\epsfxsize=10.cm \epsfbox{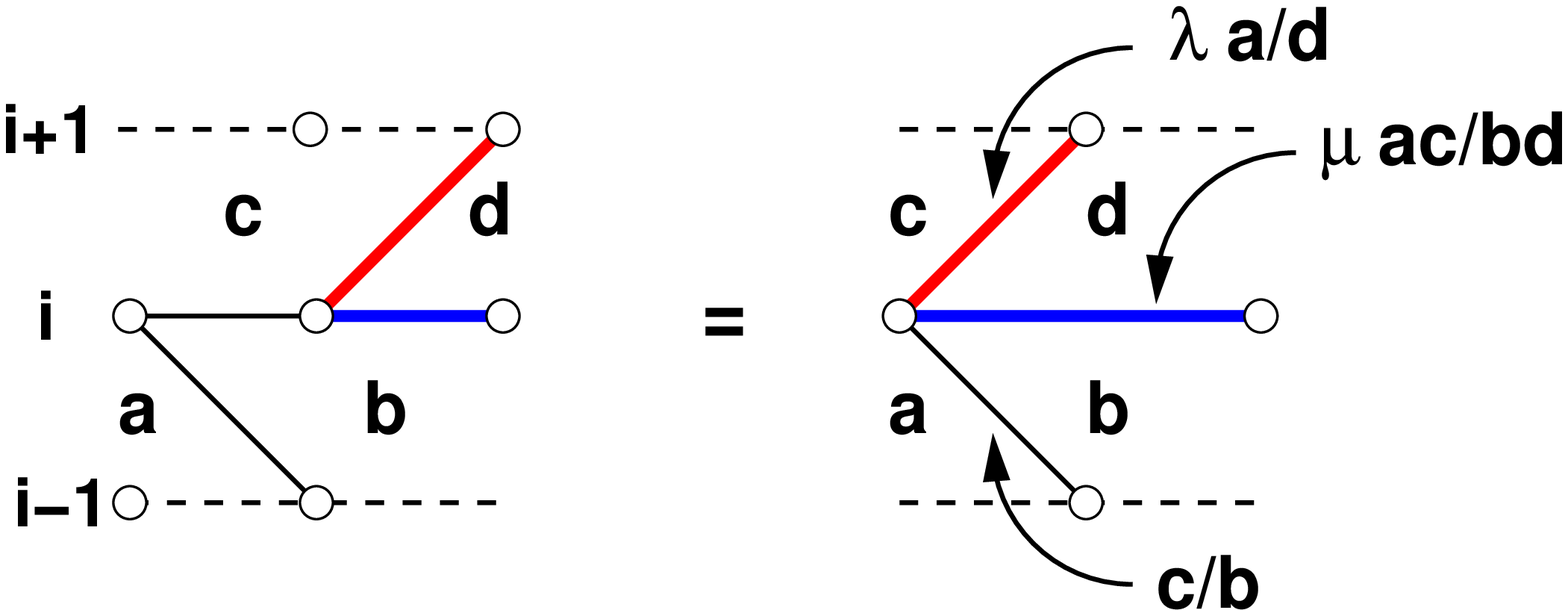}}}  \, ,$$
where we have indicated the new edge weights (and kept the color code for thickened edges
receiving extra multiplicative weights $\mu$ (blue) and $\lambda$ (red).
\end{lemma}

Applying Lemma \ref{equival} to all the pieces corresponding to products $U_{i-1}V_i$
in the network $G_{min}(n)$ allows to transform it into a 
new network ${\tilde G}_{min}(n)$ with the same network matrix, but which is now a subset 
of the directed triangular lattice (with edges oriented from left to right). For $n=4$ this gives:
$$ {\tilde G}_{min}(4)= \raisebox{-3.cm}{\hbox{\epsfxsize=7.cm \epsfbox{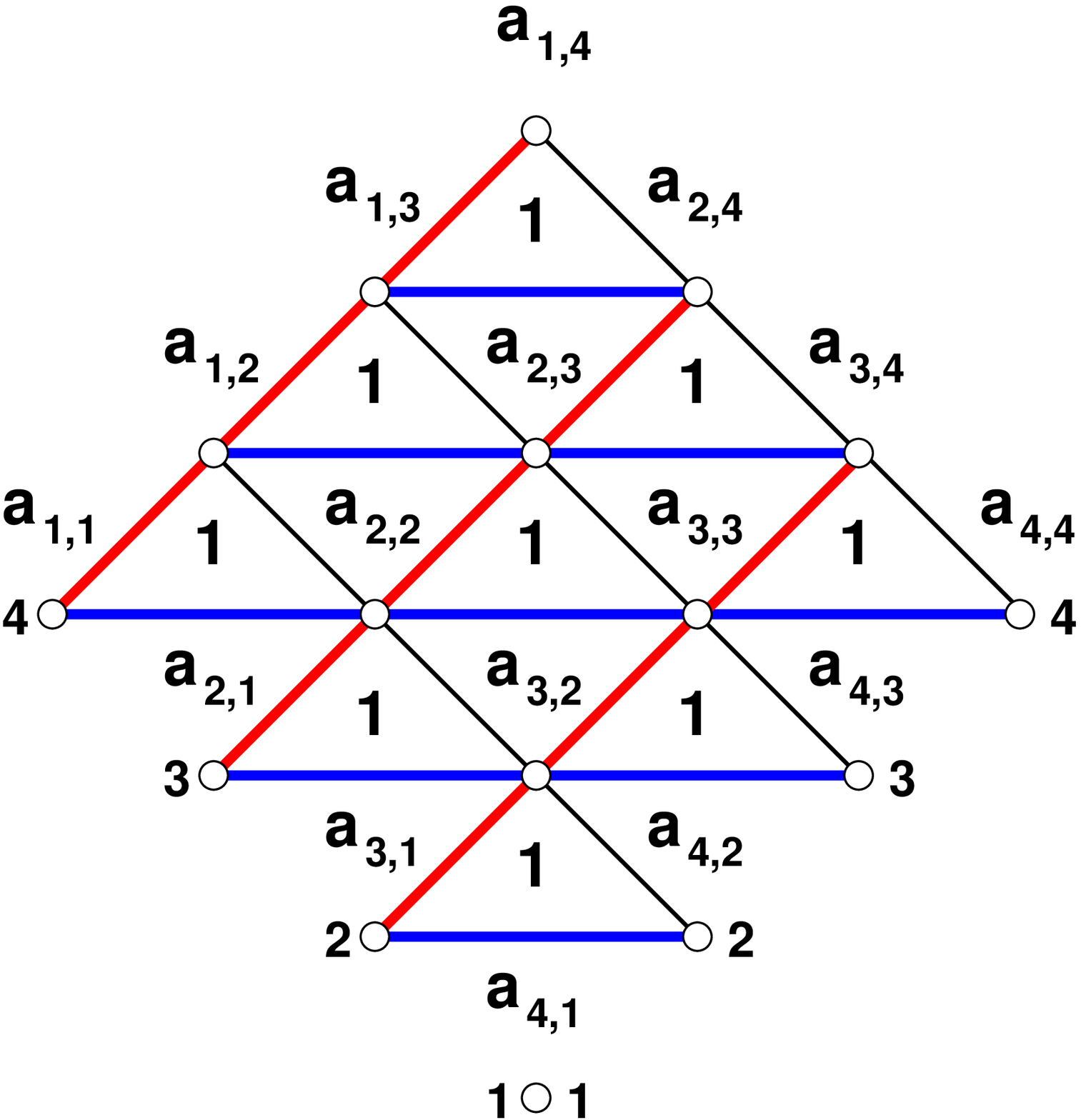}}} $$
where we have indicated the face variables corresponding to the initial data \eqref{initdat}.
Note that the face labels of all up-pointing triangles are $1$, whereas the down-pointing ones 
are the matrix elements of $A$.  Note also that for technical reasons we have included a 
bottom isolated vertex. The matrix associated with this network, ${\tilde\Theta}_{min}(n)$, 
has therefore size $2n-1\times 2n-1$.
The entry point $1$ is trivially identified with the exit point $1$.

The edge weights read as follows:
\begin{equation}\label{edgeweights} \raisebox{-1.8cm}{\hbox{\epsfxsize=14.cm \epsfbox{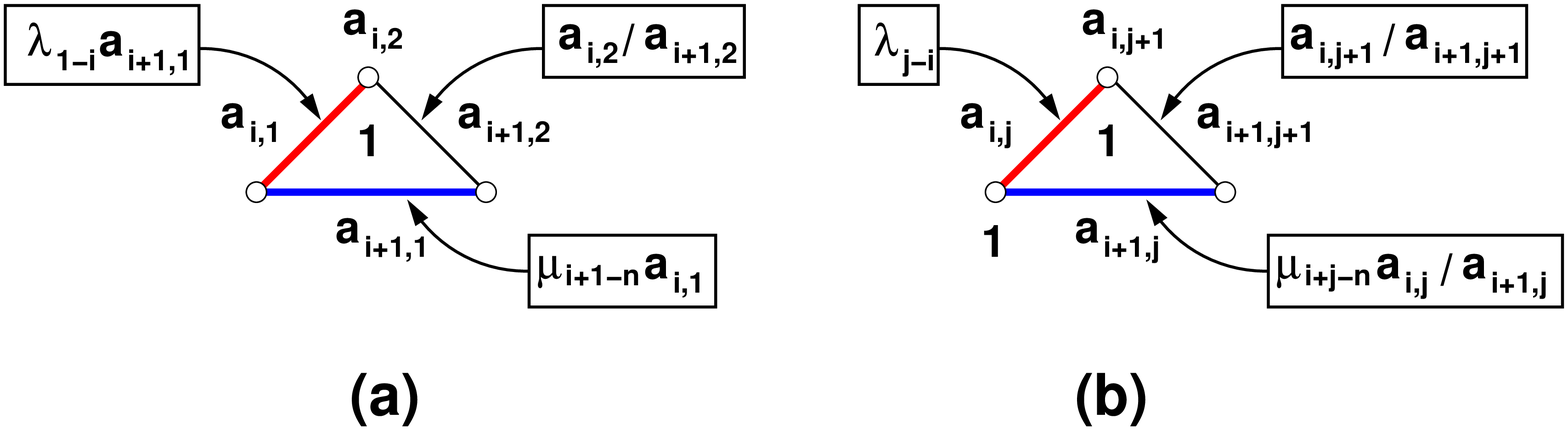}}} 
\end{equation}
depending on whether the triangle with face label $a_{i,j}$ is on the SW border (a) for $j=1$ or anywhere else 
(b) for $j>1$. In view of
the formula of Theorem \ref{ladet}, it is natural to absorb the prefactor $\prod_{i=2}^n a_{i,1}^{-1}$
into a redefinition of the network ${\tilde G}_{min}$. Indeed, dividing the weights 
of the thickened colored edges in $(a)$ by $a_{i+1,1}$, for $i=1,2,...,n-1$ exactly absorbs the prefactor,
as the paths starting at the left vertex must either traverse the red or the blue adjacent edge.
After this transformation, all the weights have the form $(b)$ of \eqref{edgeweights}
for $1\leq i,j\leq n-1$, including the ones 
on the SW border. Denoting by ${\hat G}_{min}(n)$ the network thus modified, and by ${\hat a}_{min}(n)$ 
the $n\times n$ principal minor of the associated matrix ${\hat \Theta}_{min}(n)$,
we have the following reformulation of the generalized Lambda-determinant:
\begin{thm}\label{latlam}
The generalized Lambda-determinant of $A$ is given by:
$$ \vert A\vert_{\lambda,\mu}= {\hat a}_{min}(n)\, \prod_{i=1}^n a_{n,i} $$
As such, it is the partition function of families of $n$ non-intersecting paths on the 
(directed triangular lattice) network ${\hat G}_{min}(n)$,
starting from the $n$ SW vertices and ending at the $n$ SE vertices, multiplied by all the entries of the last
row of $A$.
\end{thm}

\subsection{Another determinant formula for the generalized Lambda-determinant}
Let us now briefly describe a method for explicitly computing the generalized Lambda-determinant of a given matrix $A$.
The idea is to generate the paths eventually contributing to the minor ${\hat a}_{min}(n)$. We consider points $(x,y)$ in the
integer plane $\Z^2$, and a Hilbert space $H$ with a distinguished canonical basis $\vert x,y\rangle$, $x,y\in\Z$,
and dual basis $\langle u,v \vert$, $u,v\in\Z$ such that $\langle u,v \vert x,y\rangle=\delta_{u,x}\delta_{v,y}$.
Let us introduce operators ${\hat z},{\hat w}$ acting on $H$ as follows:
\begin{equation} {\hat w} \vert x,y\rangle = \vert x+1,y+1\rangle \qquad {\hat z} \vert x,y\rangle = \vert x+1,y-1\rangle
\end{equation}
Consider now the embedding of the network ${\hat G}_{min}(n)$ into $\Z^2$, with vertices 
$(x,y)\in {\mathcal D}_n$, with:
$${\mathcal D}_n=\{ (x,y)\in \Z^2\,  {\rm such}\,  {\rm that}\, 
x+y=0\,  {\rm mod}\,  2\,  {\rm and}\, 
|x| +|y-n+1| \leq n-1\}
$$
Note that the bottom vertex of ${\hat G}_{min}(n)$ is now at the origin $(0,0)$.
The starting points of the paths contributing to ${\hat a}_{min}(n)$ are 
$s_i=(1-i,i-1)$, $i=1,2,...,n$ and the endpoints are $e_j=(j-1,j-1)$, $j=1,2,...,n$. 
To each vertex $(x,y)\in {\mathcal D}_n$, we attach the label of the adjacent down-pointing triangle of ${\hat G}_{min}(n)$,
$\{(x-1,y+1),(x+1,y+1),(x,y)\}$. For simplicity, we denote by $\al_{x,y}=a_{\frac{x-y}{2}+n,\frac{x+y}{2}+1}$
this label. We also attach
the parameters $\ell_y=\lambda_{y+1-n}$ and $m_x=\mu_x$ respectively corresponding to the weights of the
adjacent edges $(x,y)-(x+1,y+1)$ and $(x-1,y+1)-(x+1,y+1)$ of ${\hat G}_{min}(n)$. This leads to the natural definition of the following
diagonal operators:
$$ {\hat \al} \vert x,y \rangle = \al_{x,y} \vert x,y \rangle, \quad {\hat \ell} \vert x,y \rangle=\ell_y \vert x,y \rangle, \quad 
{\hat m} \vert x,y \rangle=m_x \vert x,y \rangle$$

Starting from a vertex $(x,y)$
there are three possibilities of steps in ${\hat G}_{min}(n)$, with the weights described in \eqref{edgeweights} (b).
These weighted steps are generated by the following operator acting on $H$:

\begin{defn}\label{Top}
We define the transfer operator $T$, acting on $H$ as follows:
\begin{eqnarray*}
T&=&{\hat w} {\hat \ell}+ {\hat w}{\hat m}{\hat \al}^{-1}{\hat z} {\hat \al} +{\hat \al}^{-1}{\hat z} {\hat \al}\\
T\vert x,y \rangle &=&\ell_y \vert x+1,y+1 \rangle+m_{x+1}\frac{\al_{x,y}}{\al_{x+1,y-1}} \vert x+2,y \rangle
+\frac{\al_{x,y}}{\al_{x+1,y-1}}\vert x+1,y-1 \rangle
\end{eqnarray*}
\end{defn}

Let $Z_{i,j}(n)$ denote the partition function for paths from $s_i$ to $e_j$ on ${\hat G}_{min}(n)$.

\begin{lemma}\label{partf}
We have:
$$Z_{i,j}(n)= \langle j-1,j-1\vert (I-T)^{-1} \vert 1-i,i-1\rangle $$
or equivalently:
$$Z_{i,j}(n)= \langle0,0\vert {\hat w}^{1-j} (I-T)^{-1} {\hat z}^{1-i}\vert 0,0\rangle $$
\end{lemma}
\begin{proof}
By matching the Definition \ref{Top} with the weights of \eqref{edgeweights} (b), we easily see that
each step of a path on ${\hat G}_{min}(n)$ is generated by $T$. More precisely,
the quantity $\langle u,v\vert T^k\vert x,y\rangle$ is the sum over all
paths on ${\hat G}_{min}(n)$ of length $k$ starting at the point $(x,y)$ and ending at the point $(u,v)$, 
of the product of the corresponding network edge weights. The first part of the Lemma follows.
For the second part, we note that the adjoint $w^\dagger=w^{-1}$, and 
$w\vert x,y\rangle =\vert x+1,y+1\rangle$, so the dual is 
$$\langle x+1,y+1\vert=\langle x,y\vert {\hat w}^{\dagger}=\langle x,y\vert {\hat w}^{-1}$$
The second formula follows.
\end{proof}

Finally, combining Lemma \ref{partf} and Theorem \ref{latlam},
we obtain a compact expression for the generalized Lambda-determinant:

\begin{thm}\label{finlut}
\begin{equation} \vert A\vert_{\lambda,\mu} = \det_{1\leq i,j\leq n} \, 
\langle0,0\vert {\hat w}^{1-j}\, {\hat \al} \,\Big(I-{\hat w} {\hat \ell}
- {\hat w}{\hat m}{\hat \al}^{-1}{\hat z} {\hat \al} -{\hat \al}^{-1}{\hat z} {\hat \al}\Big)^{-1}\, {\hat z}^{1-i}\vert 0,0\rangle
\end{equation}
\end{thm}
\begin{proof} 
By definition, we have 
$${\hat a}_{min}(n)=\det_{1\leq i,j\leq n} Z_{i,j}(n)=\sum_{\sigma\in S_n}{\rm sgn}(\sigma)\, \prod_{j=1}^n Z_{\sigma(j),j}(n)$$
The formula of Theorem \ref{latlam} includes a prefactor $\prod_{j=1}^n a_{n,j}$.
Let us absorb the terms $\al_{j-1,j-1}=a_{n,j}$ in each of the factors $Z_{\sigma(j),j}(n)$, $j=1,2,...,n$. 
As $\langle 0,0\vert {\hat w}^{1-j}=\langle j-1,j-1\vert$,
we have $\al_{j-1,j-1} \langle 0,0\vert {\hat w}^{1-j}=\langle 0,0\vert {\hat w}^{1-j} {\hat \al}$. The theorem follows.
\end{proof}

\begin{remark}\label{remax}
As a non-trivial check of Theorem \ref{finlut}, let us evaluate the determinant in the homogeneous case when
$\lambda_a=\lambda$ and $\mu_a=\mu$ for all $a$, and $a_{i,j}=1$ for all $i,j$. All the operators 
${\hat z},{\hat w},{\hat \ell},{\hat m}$ commute, and ${\hat b}=I$. $Z_{i,j}(n)$ is therefore simply
the  coefficient of $z^{i-1}w^{j-1}$ of the series:
$$ f_Z(z,w)=\frac{1}{1-z-\lambda w-\mu z w} $$
where the subscript $Z$ stands for the (infinite) matrix with entries $Z_{i,j}(\infty)$. Note that $f_Z$ is independent of $n$,
which is simply the size of the principal minor of $Z$ we are supposed to compute to obtain $\vert A\vert_{\lambda,\mu}$.
We have the following (infinite size) LU decomposition:
$$ Z= B(1,1)^t\, B(\lambda+\mu,\lambda)$$
where the infinite matrices $B(\al,\beta)$ are upper triangular, with non-zero matrix elements 
$B(\al,\beta)_{i,j}={j\choose i}\al^i\beta^{j-i}$
for $j\geq i\geq 0$, or equivalently with generating functions
$$ f_{B(\al,\beta)}(z,w)=\sum_{i,j\geq 0}  B(\al,\beta)_{i,j} z^i w^j=\frac{1}{1-\beta w-\alpha z w}\, .$$
The LU decomposition is easily obtained by the convolution formula:
$$ f_{B(1,1)^t\, B(\lambda+\mu,\lambda)}(z,w)=\oint {dt \over 2i\pi t}f_{B(1,1)^t}(z,t^{-1}) f_{B(\lambda+\mu,\lambda)}(t,w)\, .$$
where the integral extracts the constant term in $t$.
Finally, by triangularity of the factors,
the product $Z=B(1,1)^t\, B(\lambda+\mu,\lambda)$, when truncated to the first $n$ rows and columns, truncates
to the product of the $n\times n$ truncations of the factors. We deduce that the principal $n\times n$ minor is:
$$\vert A\vert_{\lambda,\mu}=\det_{0\leq i,j\leq n-1} (B(1,1)_{j,i})\, \det_{0\leq i,j\leq n-1}(B(\lambda+\mu,\lambda)_{i,j})
=(\lambda+\mu)^{n(n-1)/2}$$
In agreement with the result of Example \ref{maxvan} for $a_i=1$ for all $i$.
\end{remark}

\subsection{Transformation into 6V model/ASM: proof of Theorem \ref{main}}

By Theorem \ref{latlam}, the generalized Lambda-determinant is expressed as a sum over 
configurations of $n$ non-intersecting paths on the network ${\hat G}_{min}(n)$. Any such configuration
determines entirely the state of the down-pointing triangles labelled by the matrix elements $(a_{i,j})_{i,j\in I}$
(we may draw such triangles around $a_{i,1}$, $a_{i,n}$ $a_{1,j}$ and $a_{n,j}$ as well).
We say that a vertex (resp. edge) is occupied (resp. empty) if a path of the configuration traverses it.
This gives rise to seven possible local configurations of the bottom vertex and the three edges of each
down-pointing triangle, that we associate with the local configurations of the 6V model as follows:
\begin{equation}\label{bijsixv} \raisebox{-2.cm}{\hbox{\epsfxsize=14.cm \epsfbox{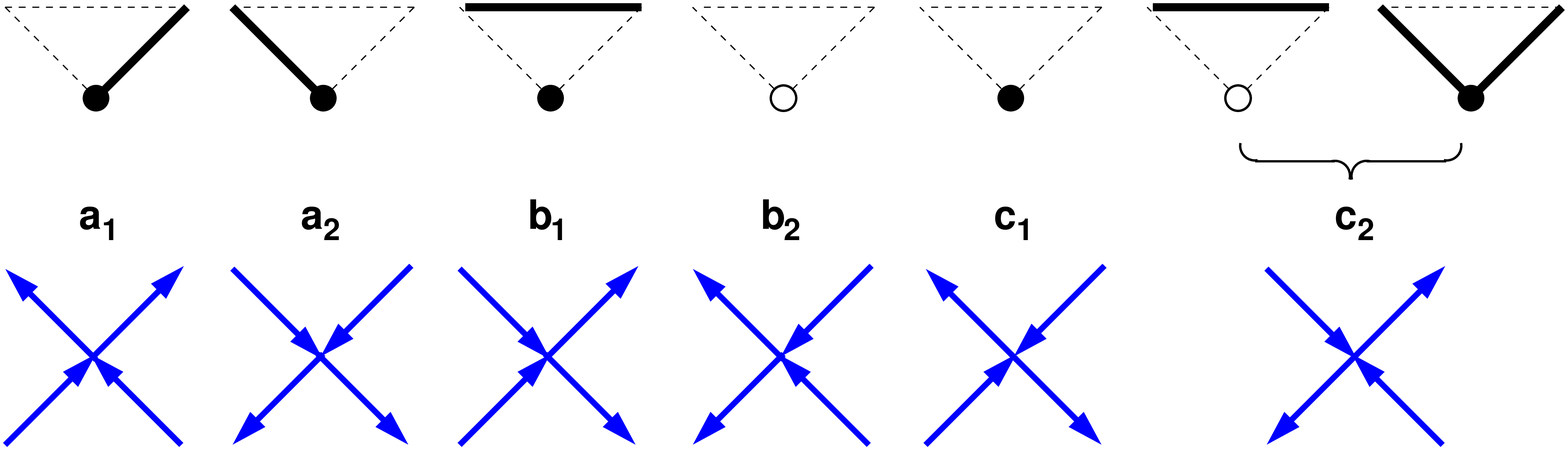}}}\end{equation}
where we have represented occupied vertices/edges by filled circles/thick solid lines, and empty ones
by empty circles/dashed lines. This is a mapping $2^m$ to $1$ of the configurations of $n$ non-intersecting
paths on ${\hat G}_{min}(n)$ with a total of $m$ down-pointing triangle configurations among the two on the right
to those of the 6V model with DWBC on the $n\times n$ grid rotated by $+\pi/4$, with $m$ vertices of type $c_2$.
The latter in turn are in bijection with ASMs of size $n\times n$ with $m$ entries equal to $-1$. This is illustrated 
in the case $n=4$, $m=1$ below:
$$  \raisebox{-1.8cm}{\hbox{\epsfxsize=16.cm \epsfbox{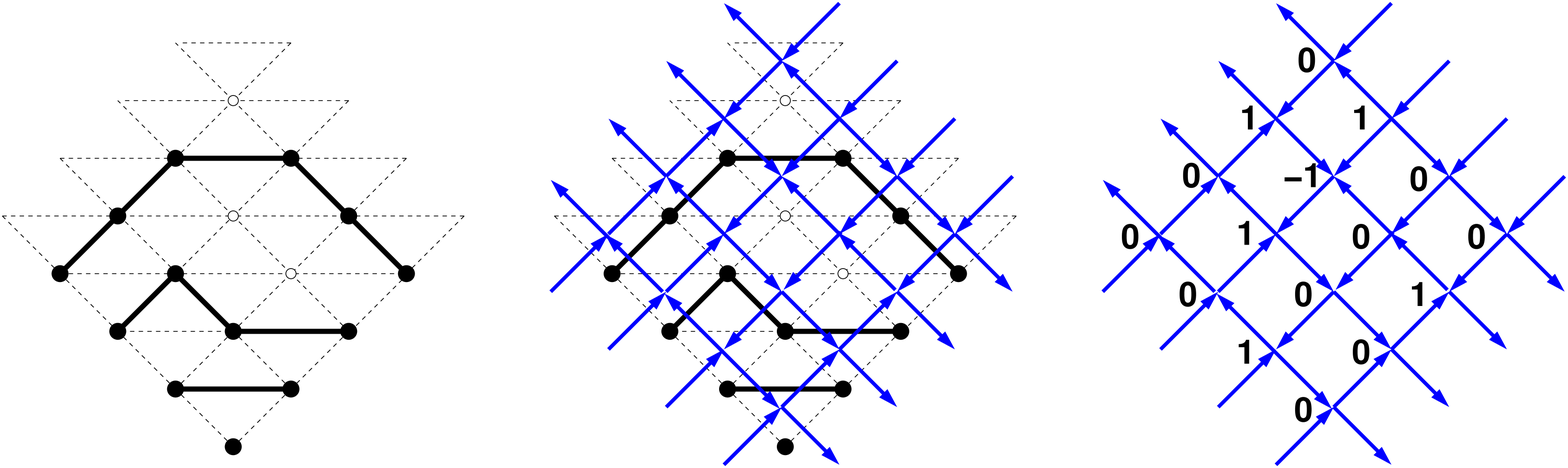}}} $$

Let us now collect the various weights of the path configurations
contributing to $\vert A\vert_{\lambda,\mu}$, by distinguishing the dependence on the 
matrix elements of $A$ (face labels)
and the coefficients $\lambda,\mu$ (edge labels).
The matrix elements $a_{i,j}$
are the labels of the down-pointing triangular faces. For each face labeled $a$, only four edges of the network
${\hat G}_{min}(n)$ have weights depending on $a$. These are:
$$  \raisebox{-1.8cm}{\hbox{\epsfxsize=5.cm \epsfbox{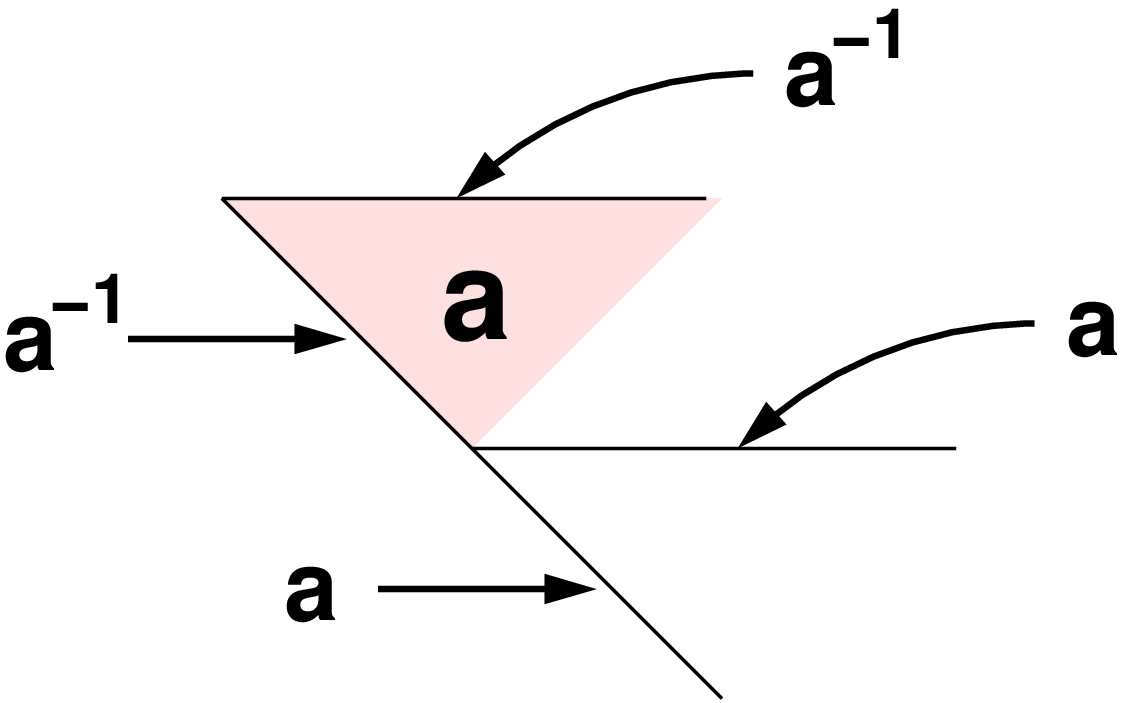}}} $$
where we have shaded the down-pointing triangle with label $a$, and indicated the $a$-dependence 
of the four relevant edges. This gives rise to the only following contributions involving the matrix element
$a$, depending on the state of occupation of the four above edges, which we list by configuration
of the down-pointing triangle:
\begin{equation}\label{fintable}  \raisebox{-1.8cm}{\hbox{\epsfxsize=14.cm \epsfbox{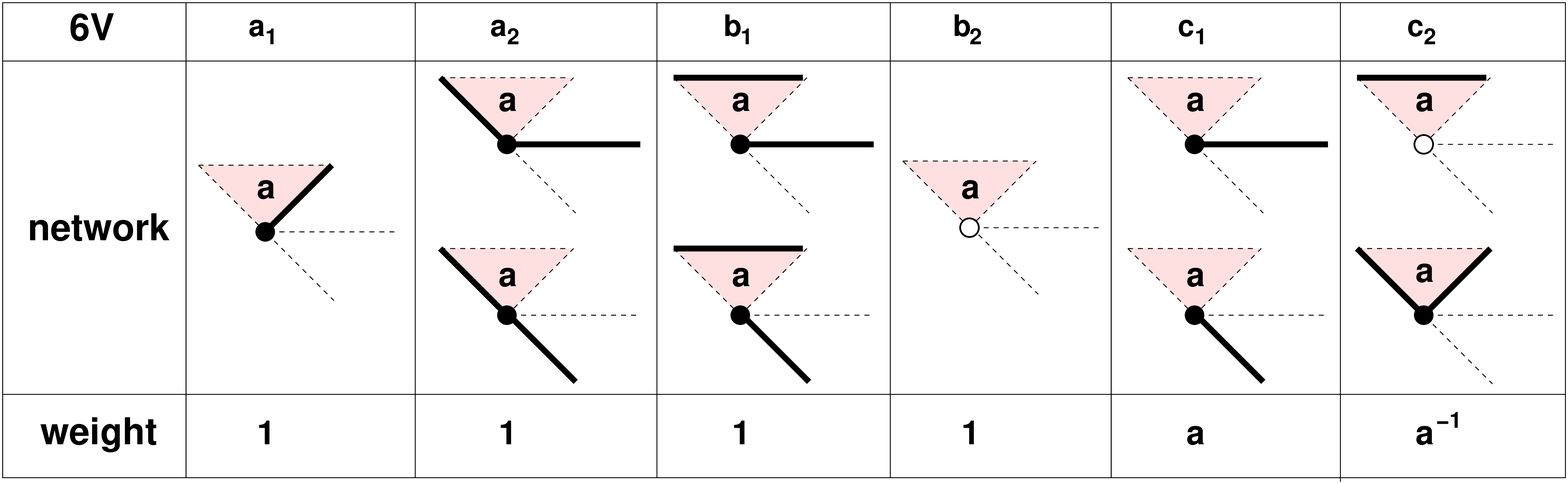}}} 
\end{equation}
The dependence on the face label $a=a_{i,j}$ is therefore $a^b$, where $b=b_{i,j}\in \{0,\pm 1\}$ is the matrix
element of the ASM in bijection with the associated 6V model configuration. This however is only valid for $i\neq n$,
and extra care should be taken with the 
SE down-pointing triangles. Indeed, their bottom vertex must always be occupied (it is an exit point for the non-intersecting
path configuration), and only two configurations may occur: (i) with the top edge occupied (6V configuration $b_1$)
or (ii) with the left diagonal edge occupied (6V configuration $a_2$). In both cases, the occupied edge receives
a weight $a^{-1}$:
$$   \raisebox{-1.8cm}{\hbox{\epsfxsize=7.cm \epsfbox{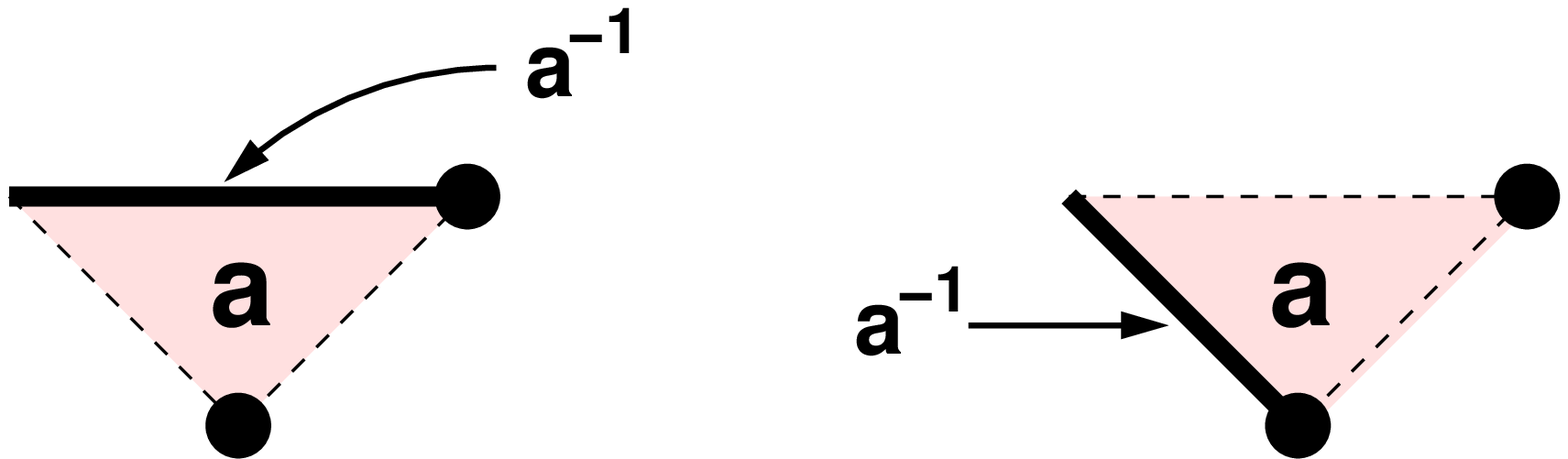}}}  $$
However, recall that the formula for the generalized Lambda-determinant of Theorem \ref{latlam}
contains a prefactor $\prod_{i=1}^n a_{n,i}$ which is the product over the labels of the SE down-pointing triangles.
We may therefore absorb this prefactor in a redefinition of the weights of horizontal and left diagonal edges 
adjacent to the SE faces, namely by multiplying them by the face label $a$. The resulting network $G'_{min}(n)$
and the $n\times n$ principal minor $a'_{min}(n)$ of its matrix are such that 
\begin{equation}\label{lutfin} \vert A\vert_{\lambda,\mu}= a'_{min}(n)\,. \end{equation}
Note that this latter transformation of the network into $G'_{min}(n)$ corresponds exactly to the
insertion of the $\hat \alpha$ operator in Theorem \ref{finlut}.
According to the table \eqref{fintable}, each configuration of $n$ non-intersecting paths on the network 
$G'_{min}(n)$ that contributes to $a'_{min}(n)$ receives a weight $\prod_{i,j} a_{i,j}^{b_{i,j}}$ where
$b_{i,j}$ is the ASM in bijection with the 6V configuration associated to the path configuration.

Finally, the dependence of the edge weights on the coefficients $\lambda_a,\mu_a$ 
is clear from \eqref{edgeweights} $(b)$:
around the down-pointing triangle with label $a_{i,j}$, the horizontal edge receives the weight $\mu_{i+j-n-1}$,
while the left diagonal edge receives the weight $\lambda_{j-i}$. This gives the total weights
\begin{equation}\label{tablefin}  \raisebox{-1.8cm}{\hbox{\epsfxsize=14.cm \epsfbox{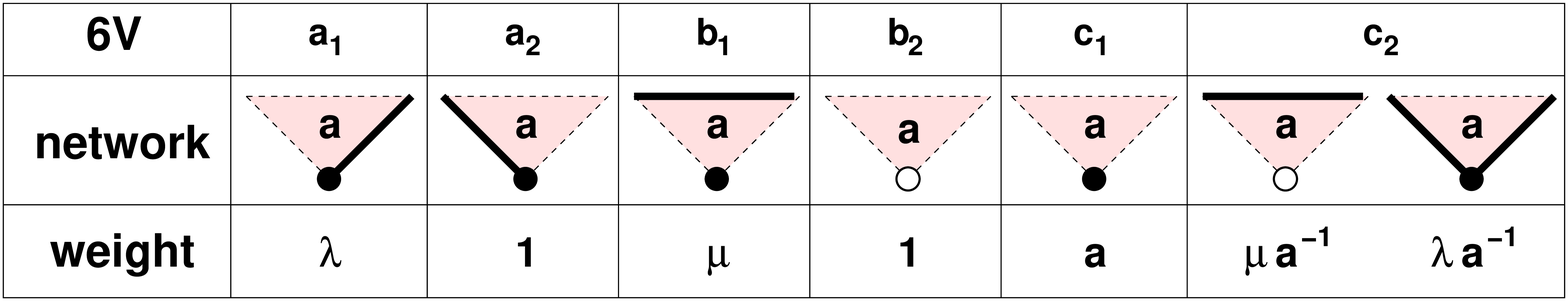}}} 
\end{equation}
where for $a=a_{i,j}$, $\lambda=\lambda_{j-i}$ and $\mu=\mu_{i+j-n-1}$, for all $i,j=1,2,...,n$.

Starting from the expression \eqref{lutfin}, we may rearrange the configurations of paths with a total of $m$
triangle configurations of type $c_2$ say in positions $(i_1,j_1),(i_2,j_2),...,(i_m,j_m)$ (and a total of $m+n$ 
triangle configurations of type $c_1$ in fixed positions)
into a $2^m$-tuple, in bijection with  the unique ASM with exactly $m$ entries $-1$ at the same positions
(and $m+n$ entries equal to $1$ at the same positions).
Summing over the weights for such a $2^m$-tuple yields an overall factor of
$$\prod_{r=1}^m \left(\lambda_{j_r-i_r} +\mu_{i_r+j_r-n-1} \right) $$
according to the table \eqref{tablefin}. Collecting all the other weights,
we may rewrite $\vert A\vert_{\lambda,\mu}$ as a sum over $n\times n$ ASMs $B$, with the weights of the form 
$\prod_{i,j\in I} v_{i,j}(a_{i,j},b_{i,j}\vert\lambda_{j-i},\mu_{i+j-n-1})$ where
$$v_{i,j}(a,b\vert\lambda,\mu) =a^b \times \left\{ \begin{matrix} 
\lambda & {\rm if}\ {\mathcal C}_{i,j}(B)=a_1\\
\mu & {\rm if}\ {\mathcal C}_{i,j}(B)=b_1\\
\lambda+\mu & {\rm if}\ {\mathcal C}_{i,j}(B)=c_2\\
1 & {\rm otherwise} 
\end{matrix}\right. $$
This completes the proof of Theorem \ref{main}.

\section{Discussion and conclusion}

In this paper, we have defined the generalized Lambda-determinant of any square matrix $A$. Via Theorem\ref{main}
this quantity may be reexpressed in terms of weighted configurations of the 6V-DWBC model on a square grid
of the same size as $A$, in bijection with ASMs of the same size.

\subsection{Properties}

Let us show how the formula of Theorem \ref{main} may be applied to give an alternative
proof of Proposition \ref{apropos}. Notations are as in Proposition \ref{apropos}.

We simply have to examine how the 6V/ASM configurations contributing to the formula
\eqref{formu}  are changed when we change $A\to \varphi(A)$,
namely (i) rotate the matrix $A$ by a clockwise quarter-turn $-\pi/2$ ($\varphi=\sigma$) or (ii) transpose the
matrix $A$, namely reflect it w.r.t. to its diagonal ($\varphi=\tau$). 
We have the following action of $\varphi$ on the 6V-DWBC
configurations $\mathcal C$, compatible with the transformation
$B\to \varphi(B)$ for any $n\times n$ ASM $B$:

\begin{lemma}
Let $\varphi\in\{\sigma,\tau\}$ act on configurations $\mathcal C$ of the $n\times n$ grid 6V-DWBC
model as follows:
$\varphi({\mathcal C})$ is obtained by applying directly the transformation $\varphi$ to the arrow configuration
(rotation by $-\pi/2$ for $\varphi=\sigma$, reflection w.r.t. the diagonal for $\varphi=\tau$), and then
performing a global flip of all the edge orientations. Then we have for all $n\times n$ ASM $B$:
$$ \varphi\Big({\mathcal C}(B)\Big)={\mathcal C}\Big(\varphi(B)\Big)\, .$$
\end{lemma}
\begin{proof}
Note first that the global sign flip reinstates the correct orientations of external edges to restore 
DWBC boundary conditions. Under the transformation $\varphi$, the six vertex configurations are
mapped as follows:
\begin{equation}\label{transfos} \begin{matrix}
\sigma(a_1)=b_2 & \sigma(b_1)=a_1 & \sigma(a_2)=b_1 & \sigma(b_2)=a_2
& \sigma(c_1)=c_1 & \sigma(c_2)=c_2 \\
\tau(a_1)=a_1& \tau(b_1)=b_2& \tau(a_2)=a_2& \tau(b_2)=b_1
& \tau(c_1)=c_1& \tau(c_2)=c_2
\end{matrix}
\end{equation}
As both $c_1$ and $c_2$ are invariant under $\varphi\in \{\sigma,\tau\}$, the ASM associated to the
configuration $\varphi({\mathcal C})$ is simply $\varphi(B)$, and the Lemma follows.
\end{proof}

Let us now evaluate $\vert \varphi(A)\vert_{\varphi(\lambda),\varphi(\mu)}$.
Using the formula \eqref{formu},
the weight $w$ reads:
$$ w_{i,j}\Big(\varphi(A),B;\varphi(\lambda),\varphi(\mu)\Big)=\Big(\varphi(A)_{i,j}\Big)^{b_{i,j}}
\times \left\{ \begin{matrix}
\varphi(\lambda)_{j-i} & {\rm if} \, {\mathcal C}(B)_{i,j}=a_1 \\
\varphi(\mu)_{i+j-n-1} & {\rm if} \, {\mathcal C}(B)_{i,j}=b_1 \\
\varphi(\lambda)_{j-i} +\varphi(\mu)_{i+j-n-1} & {\rm if}\, {\mathcal C}(B)_{i,j}=c_2 \\
1 & {\rm otherwise}
\end{matrix}\right.
$$
Let us perform the change of summation variables $B=\varphi(C)$. For $\varphi=\sigma,\tau$ respectively, this
gives weights:
\begin{eqnarray*}
w_{i,j}\Big(\sigma(A),\sigma(C);\sigma(\lambda),\sigma(\mu)\Big)&=&\Big(a_{n+1-j,i}\Big)^{c_{n+1-j,i}}
\times \left\{ \begin{matrix}
\mu_{i-j} & {\rm if} \, {\mathcal C}(C)_{n+1-j,i}=b_1 \\
\lambda_{i+j-n-1} & {\rm if} \, {\mathcal C}(C)_{n+1-j,i}=a_2 \\
\mu_{i-j} +\lambda_{i+j-n-1} & {\rm if}\, {\mathcal C}(C)_{n+1-j,i}=c_2 \\
1 & {\rm otherwise}
\end{matrix}\right. \\
w_{i,j}\Big(\tau(A),\tau(C);\tau(\lambda),\tau(\mu)\Big)&=&\Big(a_{j,i}\Big)^{c_{j,i}}
\times \left\{ \begin{matrix}
\lambda_{i-j} & {\rm if} \, {\mathcal C}(C)_{j,i}=a_1 \\
\mu_{i+j-n-1} & {\rm if} \, {\mathcal C}(C)_{j,i}=b_2 \\
\lambda_{i-j} +\mu_{i+j-n-1} & {\rm if}\, {\mathcal C}(C)_{j,i}=c_2 \\
1 & {\rm otherwise}
\end{matrix}\right.
\end{eqnarray*}
Equivalently, we may write:
\begin{eqnarray*}
w_{j,n+1-i}\Big(\sigma(A),\sigma(C);\sigma(\lambda),\sigma(\mu)\Big)&=&\Big(a_{i,j}\Big)^{c_{i,j}}
\times \left\{ \begin{matrix}
\mu_{i+j-n-1} & {\rm if} \, {\mathcal C}(C)_{i,j}=b_1 \\
\lambda_{j-i} & {\rm if} \, {\mathcal C}(C)_{i,j}=a_2 \\
\lambda_{j-i} +\mu_{i+j-n-1} & {\rm if}\, {\mathcal C}(C)_{i,j}=c_2 \\
1 & {\rm otherwise}
\end{matrix}\right. \\
w_{j,i}\Big(\tau(A),\tau(C);\tau(\lambda),\tau(\mu)\Big)&=&\Big(a_{i,j}\Big)^{c_{i,j}}
\times \left\{ \begin{matrix}
\lambda_{j-i} & {\rm if} \, {\mathcal C}(C)_{i,j}=a_1 \\
\mu_{i+j-n-1} & {\rm if} \, {\mathcal C}(C)_{i,j}=b_2 \\
\lambda_{j-i} +\mu_{i+j-n-1} & {\rm if}\, {\mathcal C}(C)_{i,j}=c_2 \\
1 & {\rm otherwise}
\end{matrix}\right.
\end{eqnarray*}
These two expressions match $w_{i,j}(A,C;\lambda,\mu)$ provided the 6V weights 
in the definition are
changed into respectively (i) $a_1\to a_2$ and
(ii) $b_1\to b_2$. 

Proposition \ref{apropos} follows from the fact that the formula  \eqref{formu}
remains valid if we change $a_1\to a_2$ and/or
$b_1\to b_2$ in the definition of $w$. This in turn is the consequence of the following:

\begin{figure}
\centering
\includegraphics[width=14.cm]{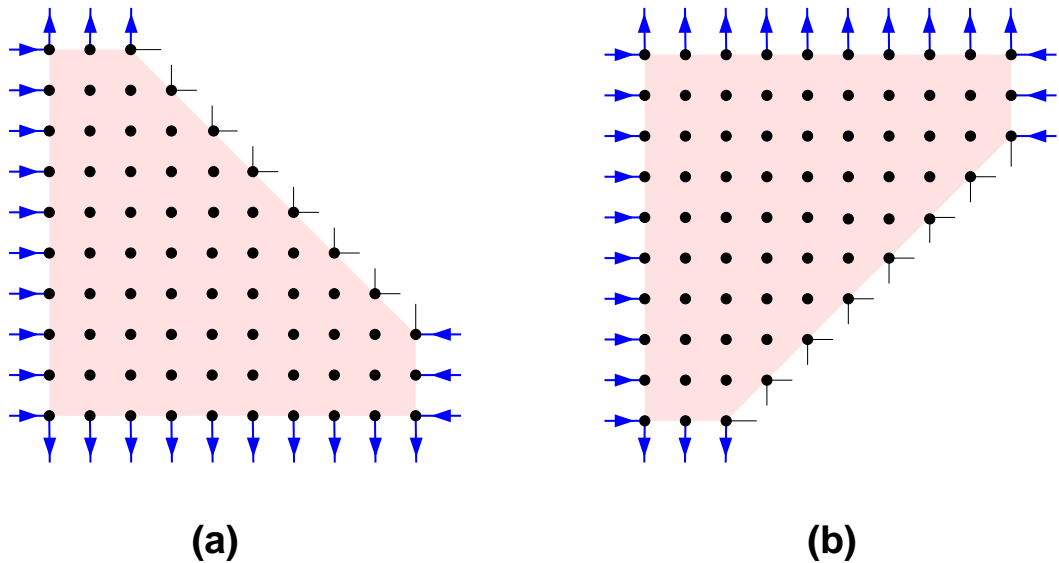}
\caption{\small The sub-configuration of a 6V-DWBC configuration $\mathcal C$ situated
(a) under a diagonal (b) above an anti-diagonal.
We have marked the edges of the diagonal border, respectively equal to the N and E 
edges (a) and S and E edges (b).}
\label{fig:diaganti}
\end{figure}

\begin{lemma}
For any 6V-DWBC configuration $\mathcal C$, we have for all $a$:
\begin{eqnarray*} 
\left\vert \{i\in \{1,2,...,n\}\, {\rm such}\, {\rm that}\, {\mathcal C}_{i,i+a}=a_1\}\right\vert &=&
\left\vert \{i\in \{1,2,...,n\}\, {\rm such}\, {\rm that}\, {\mathcal C}_{i,i+a}=a_2\}\right\vert \\
\left\vert \{i\in \{1,2,...,n\}\, {\rm such}\, {\rm that}\, {\mathcal C}_{i,a-j}=b_1\}\right\vert &=&
\left\vert \{i\in \{1,2,...,n\}\, {\rm such}\, {\rm that}\, {\mathcal C}_{i,a-j}=b_2\}\right\vert
\end{eqnarray*}
in other words there are as many $a_1$ configurations as $a_2$ along each parallel to the diagonal,
and there are as many $b_1$ configurations as $b_2$ along each parallel to the anti-diagonal.
\end{lemma}
\begin{proof}
Let us consider a diagonal $(i,i+a)_{1\leq i\leq n-a}$ of $\mathcal C$, and the sub-configuration 
situated under the line $(i,i+a-1)$ (see Fig.\ref{fig:diaganti} (a)). 
This configuration has entering edges on its left 
and possibly right vertical borders, and outgoing edges along its bottom and possibly top horizontal
borders, in equal numbers. By the neutrality condition at each vertex (there are as many entering
as outgoing edges), we deduce that there are as many entering as outgoing edges along the diagonal border.
These are the N and E edges of the vertices of the diagonal border. Inspecting the vertices on \eqref{sixvertices}
we see that all vertices have opposite orientations of their N and E edges, except for $a_1$ and $a_2$
who have respectively two outgoing and two entering N and E edges. We deduce that these occur 
in the same number to guarantee that as many edges are entering and outgoing along the diagonal border.
The argument may be repeated for anti-diagonals by considering the sub-configuration situated above the
anti-diagonal $(i,a+1-i)$ (see Fig.\ref{fig:diaganti} (b)). In that case, we must inspect the S and E edges of the
vertices of \eqref{sixvertices}, and we note that only $b_1$ and $b_2$ have S and E edges pointing both out or in,
therefore they must be in equal number along the diagonal border. The Lemma follows.
\end{proof}

\subsection{The generalized Lambda-determinant in terms of ASM}

Theorem \ref{main} expresses the generalized Lambda-determinant in terms of 6V-DWBC configurations.
Let us reexpress this quantity purely in ASM terms.
To this end, we must better explain the inversion number of any ASM $B$.
The contributions to ${\rm Inv}(B)=\sum_{i<j\quad k<\ell} b_{j,k}b_{i,\ell}$ 
are obtained from rectangular submatrices of $B_{[i,j]}^{[k,\ell]}$
of $B$, obtained by retaining only the intersection between rows $(i,i+1,...,j)$ and columns $(k,k+1,...,\ell)$ in $B$.
To produce a non-trivial contribution to ${\rm Inv}(B)$ the rectangle must have non-zero elements in position 
$(j,k)$ (SW) and $(i,\ell)$ (NE). Let us now concentrate on the element in position $(i,k)$ (NW). Let $h,v$ denote
respectively the first non-zero element of $B$ on row $i$ to the right of $b_{i,k}$ and on column $k$ below $b_{i,k}$.
The simplest way of thinking about the bijection from ASM to 6V-DWBC is that along rows edges are oriented from
the entries $-1$ to the entries $1$ and get reversed at the first and last $1$, with opposite rules along columns.
Assuming $b_{i,k}=0$, this gives the following 4 possibilities: 
(i) $(h,v)=(1,1)$: then ${\mathcal C}(B)_{i,k}=a_1$; 
(ii) $(h,v)=(-1,1)$: then ${\mathcal C}(B)_{i,k}=b_2$; 
(iii) $(h,v)=(1,-1)$: then ${\mathcal C}(B)_{i,k}=b_1$;
(iv) $(h,v)=(-1,-1)$: then ${\mathcal C}(B)_{i,k}=a_2$.
When computing contributions to ${\rm Inv}(B)$, we may first fix the NW corner of rectangles we sum over, and finally
sum over the position $(i,k)$ of that corner.
In cases (ii-iv), for fixed $(i,k)$ the sum gives trivially zero, as there are as many $1$'s and $-1$ either on the row or
the column of $(i,k)$ or both. So only the case (i) contributes a total of $+1$ to ${\rm Inv}(B)$. Moreover,
if $b_{i,k}\neq 0$, only cases (i) and (iv) survive, with respectively $b_{i,k}=-1$ and $b_{i,k}=1$. However for the same reason
as before, the case (iv) does not contribute to ${\rm Inv}(B)$ when we sum over the rectangles with fixed NW corner $b_{i,k}$.
We are left with only the case (i), for which ${\mathcal C}(B)_{i,j}=c_2$, and the sum over rectangles contributes a total of $+1$
to ${\rm inv}(B)$.

To summarize, terms contributing $+1$ to ${\rm Inv}(B)$ may be associated bijectively to the 
$a_1$ and $c_2$ configurations in ${\mathcal C}(B)$. 
So we have ${\rm Inv}(B)-\#(-1)_B=\#(a_1)$, the total number of $a_1$ vertices in ${\mathcal C}(B)$.
Moreover, the vertices equal to $a_1$ in ${\mathcal C}(B)$ correspond
bijectively to the entries $0$ of $B$, such that the first non-zero elements in 
$B$ in their row to the right and in their column below are both $1$.

We may repeat the above analysis with the matrix $\sigma(B)$, rotated by $-\pi/2$. What plays the role of $a_1$
after rotation are the vertices $b_1$ before rotation (see previous subsection), and $c_2$'s remain unchanged.
We deduce that ${\rm Inv}(\sigma(B))-\#(-1)_B=\#(b_1)$, the total number of $b_1$ vertices in ${\mathcal C}(B)$.
Moreover the vertices equal to $b_1$ in ${\mathcal C}(B)$  correspond
bijectively to the entries $0$ of $B$, such that the first non-zero elements in 
$B$ in their row to the right and in their column above are both $1$.

\begin{defn}
For each fixed diagonal $\delta_a(B)=\{b_{i,j}\, |\,  j-i=a\}$ for $2-n\leq a\leq n-2$, 
let us denote by ${\rm I}_a(B)$ the total number of entries $0$ 
such that the first non-zero entries to the right and below are both $1$. For each fixed anti-diagonal 
$\alpha_b=\{b_{i,j}\, |\,  i+j=b\}$ for $2-n\leq b\leq n-2$ let ${\rm I}'_b(B)$ denote the total number of 
entries $0$ such that the 
first non-zero entries to the right and above
are both $1$. Finally, let $\#(-1)_B(i,j)=\delta_{b_{i,j},-1}$.
\end{defn}

We have the following generalized Lambda-determinant expression:
\begin{thm}
The generalized Lambda-determinant of any $n\times n$ matrix $A$ reads:
$$ \vert A\vert_{\lambda,\mu}= \sum_{n\times n\, ASM\, B} \, 
\prod_{a=2-n}^{n-2} \lambda_a^{{\rm I}_a(B)} \,\mu_a^{{\rm I}'_a(B)}
\, \prod_{i,j=1}^n a_{i,j}^{b_{i,j}}\,  (\lambda_{j-i}+\mu_{i+j-n-1})^{\#(-1)_B(i,j)}$$
\end{thm}

\subsection{The general solution of the $T$-system with coefficients}

In this paper, we have concentrated on special initial conditions for the $T$-system \eqref{inhomtsys},
namely such that $T_{i,j,0}=1$ and $T_{i,j,1}=a_{\frac{j-i+n+1}{2},\frac{i+j+n+1}{2}}$. Starting from 
the general $T$-system solution of Theorem \ref{soltsys}, where {\it both} $T_{i,j,0}$ and $T_{i,j,1}$
are assigned arbitrary initial values $t_{i,j}=T_{i,j,\epsilon_{i,j,n}}$, we may repeat the steps we took to modify
the initial network $G_{min}(n)$ into a subset of the directed triangular lattice. The resulting network
(which we still denote by ${\hat G}_{min}(n)$, with associated minor ${\hat a}_{min}(n)$ 
by a slight abuse of notation) has the following edge weights
around each up-pointing triangle:
\begin{equation}\label{Trule} 
\raisebox{-1.8cm}{\hbox{\epsfxsize=8.cm \epsfbox{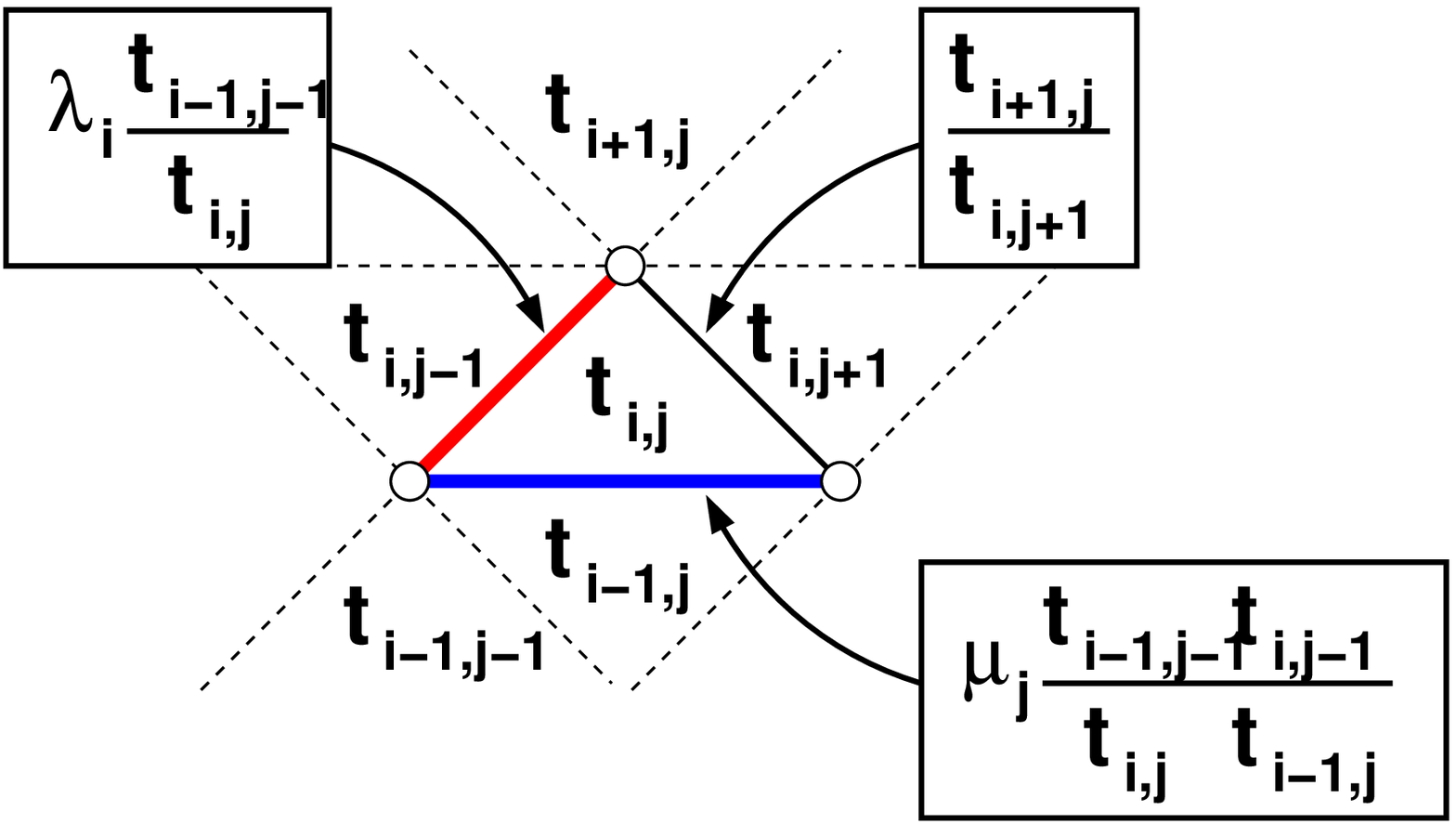}}}
\end{equation}
and we have
$$ T_{0,0,n}={\hat a}_{min}(n)\, \prod_{j=1}^n t_{j-n,j-1} =a'_{min}(n)$$
where the last transformation to $G'_{min}(n)$ consists in absorbing the $t_{j-n,j-1}$ in the SE horizontal and diagonal
edge weights adjacent to the endpoints of the paths. Noting that the initial values $T_{i,j,1}$ are the
face labels of the down-pointing triangles, while the $T_{i,j,0}$ are those of the up-pointing triangles,
we may summarize the dependence of the edge weights on the face labels of the down-pointing and
up-pointing triangles as follows:
$$ \raisebox{-1.8cm}{\hbox{\epsfxsize=10.cm \epsfbox{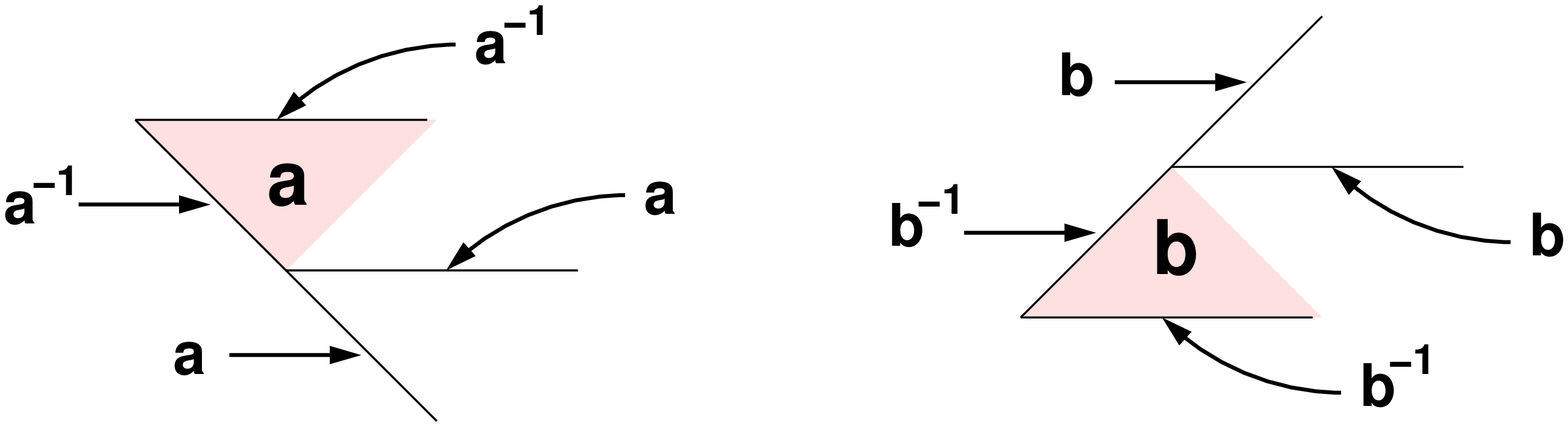}}}  $$
where $a,b=t_{i,j}$ for some $i,j$. Combining these the weights of down-pointing triangles \eqref{fintable}, 
we note that the weights acquire a simpler form if we transfer the $\lambda$ and $\mu$ parameters from the 
weights of down-pointing triangles to those of up-pointing ones, resulting in:
\begin{equation}\label{triconf} \raisebox{-1.8cm}{\hbox{\epsfxsize=14.cm \epsfbox{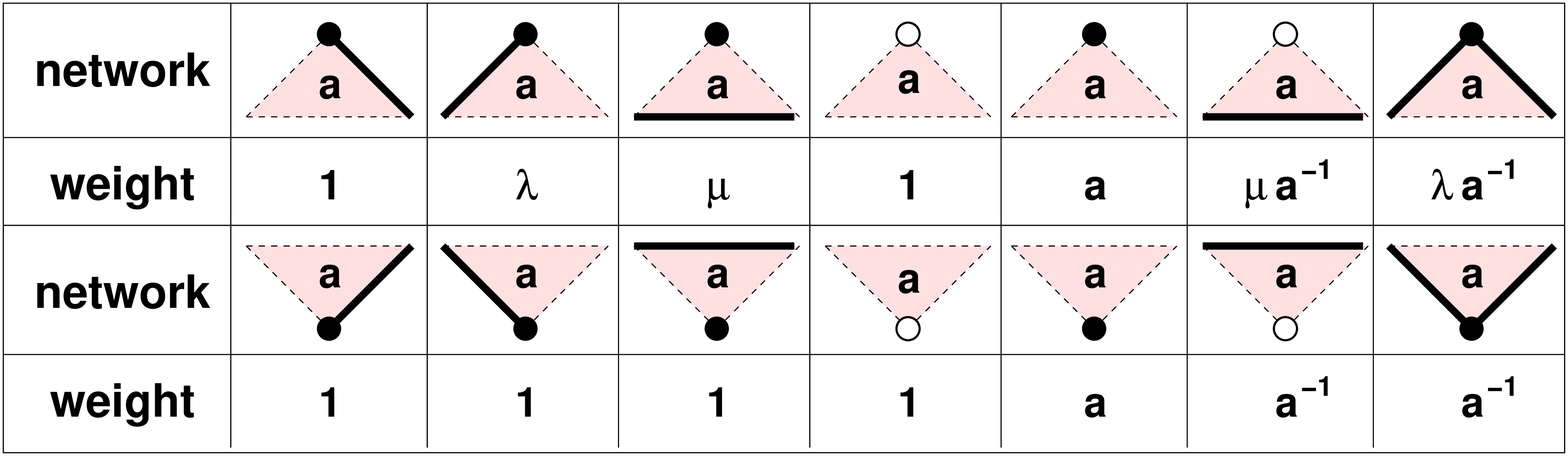}}}  
\end{equation}
where for $a=t_{i,j}$, we have $\lambda=\lambda_i$ and $\mu=\mu_{j}$.

\begin{remark}
We have associated a 6V-DWBC configuration of size $n$ to each configuration of $n$ non-intersecting paths
on $G'_{min}(n)$ from SW to SE, by the local rules \eqref{bijsixv} associating 6V vertices to each down-pointing
triangle configuration. We could have done the same for up-pointing triangles as well, by using the same 
rules on the reflected up-pointing triangles w.r.t. their horizontal edge. More precisely, this would yield a 6V-DWBC'
configuration of size $n-1$ where the prime refers to the fact that the boundary condition has opposite convention
(horizontal external edges point out of the grid, while vertical ones point into the grid). Flipping all edge orientations,
this becomes a 6V-DWBC configuration of size $n-1$. So each configuration of non-intersecting paths on $G'_{min}(n)$
that contribute to $a'_{min}(n)$ give rise naturally to a pair $(B,C)$ of ASMs of respective sizes $n$ and $n-1$.
Such pairs are called compatible pairs in the literature. However, these pairs do not play any role in our fully
inhomogeneous $T$-system solution, as weights are attached only to such compatible pairs and cannot be
disentangled nicely. In particular there is no resummation of the non-intersecting lattice path weights into weights
only pertaining to the ASMs, as opposed to the case where $T_{i,j,0}=1$ of Theorem \ref{main}, 
where the second ASM plays no role.
\end{remark}

\begin{example}
We consider $n=3$. We have the initial data assignments (represented in the $(i,j)$ plane):
$$\begin{matrix}
& & T_{2,0,1}=t_{2,0} & & \\
& T_{1,-1,1}=t_{1,-1} & T_{1,0,0}=t_{1,0} & T_{1,1,1}=t_{1,1} & \\
T_{0,-2,1}=t_{0,-2} & T_{0,-1,0}=t_{0,-1} & T_{0,0,1}=t_{0,0}& T_{0,1,0}=t_{0,1}& T_{0,2,1}=t_{0,2}\\
& T_{-1,-1,1}=t_{-1,-1} & T_{-1,0,0}=t_{-1,0} & T_{-1,1,1}=t_{-1,1} & \\
& & T_{-2,0,1}=t_{-2,0} & & 
\end{matrix}
$$
There are 8 distinct configurations of the 3 non-intersecting paths on $G'_{min}(3)$, contributing to $a'_{min}(3)$
(we have indicated the by small squares the centers of the $n^2=9$ down-pointing triangles 
and of the $(n-1)^2=4$ up-pointing ones, corresponding to labels $t_{i,j}$ with $(i,j)\in D_{0,0,3}$, on the first configuration (a)):
$$  \raisebox{-1.cm}{\hbox{\epsfxsize=16.cm \epsfbox{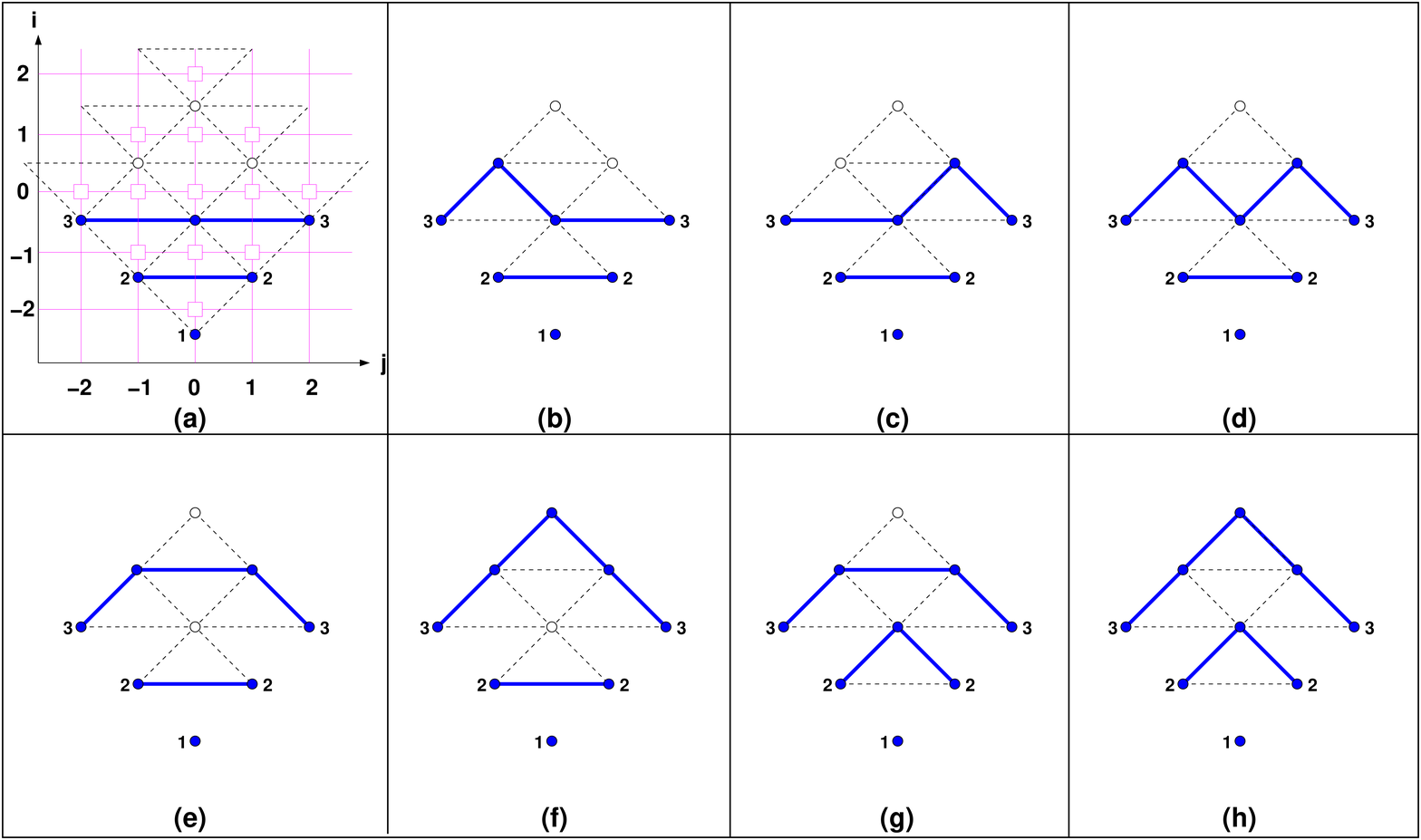}}} $$
The corresponding weights read:
$$\begin{matrix}
w_{(a)}=\mu_{-1}\mu_0\mu_1 \, \frac{t_{0,-2},t_{0,0},t_{0,2}}{t_{0,-1}t_{0,1}}& 
w_{(b)}=\lambda_0\mu_0\mu_1\, \frac{t_{-1,-1}t_{1,-1} t_{0,2}}{t_{0,-1}t_{0,1}} & 
w_{(c)}=\mu_{-1}\mu_0\lambda_0\,  \frac{t_{0,-1}t_{-1,1}t_{1,1}}{t_{0,-1}t_{0,1}}  \\
\, \, \, w_{(d)}=\lambda_0^2\mu_0  \,\frac{t_{-1,-1}t_{-1,1}t_{1,-1}t_{1,1}}{t_{0,-1}t_{0,1}t_{0,0}} & 
\, \, \, \,w_{(e)}=\lambda_0\mu_0^2\, \frac{t_{-1,-1}t_{-1,1}t_{1,-1}t_{1,1}}{t_{-1,0}t_{1,0}t_{0,0}}
& w_{(f)}=\mu_0\lambda_0\lambda_1\,  \frac{t_{-1,-1}t_{-1,1} t_{2,0}}{t_{-1,0}t_{1,0}}  \\
w_{(g)}=\lambda_{-1}\lambda_0\mu_0\,  \frac{t_{1,-1}t_{1,1} t_{-2,0}}{t_{-1,0}t_{1,0}}& 
\, \, \, \, \,w_{(h)}= \lambda_{-1}\lambda_0\lambda_1\, \frac{t_{-1,-1}t_{-1,1} t_{-2,0}}{t_{-1,0}t_{1,0}} &
\end{matrix} $$
The corresponding solution of the $T$-system is:
\begin{eqnarray*} T_{0,0,3}&=&\mu_{-1}\mu_0\mu_1 \, \frac{t_{0,-2},t_{0,0},t_{0,2}}{t_{0,-1}t_{0,1}}
+\lambda_0\mu_0\mu_1\, \frac{t_{-1,-1}t_{1,-1} t_{0,2}}{t_{0,-1}t_{0,1}}
+\mu_{-1}\mu_0\lambda_0\,  \frac{t_{0,-1}t_{-1,1}t_{1,1}}{t_{0,-1}t_{0,1}}\\
&&+\lambda_0^2\mu_0  \,\frac{t_{-1,-1}t_{-1,1}t_{1,-1}t_{1,1}}{t_{0,-1}t_{0,1}t_{0,0}}
+\lambda_0\mu_0^2\, \frac{t_{-1,-1}t_{-1,1}t_{1,-1}t_{1,1}}{t_{-1,0}t_{1,0}t_{0,0}}
+\mu_0\lambda_0\lambda_1\,  \frac{t_{-1,-1}t_{-1,1} t_{2,0}}{t_{-1,0}t_{1,0}}\\
&&+\lambda_{-1}\lambda_0\mu_0\,  \frac{t_{1,-1}t_{1,1} t_{-2,0}}{t_{-1,0}t_{1,0}}
+\lambda_{-1}\lambda_0\lambda_1\, \frac{t_{-1,-1}t_{-1,1} t_{-2,0}}{t_{-1,0}t_{1,0}}
\end{eqnarray*}
If we take  $t_{\pm 1,0}=t_{0,\pm 1}=1$, we recover the formula of Theorem \ref{main}:
\begin{eqnarray*} T_{0,0,3}&=& \mu_{-1}\mu_0\mu_1 \, {t_{0,-2},t_{0,0},t_{0,2}}
+\lambda_0\mu_0\mu_1\, {t_{-1,-1}t_{1,-1} t_{0,2}}
+\mu_{-1}\mu_0\lambda_0\, {t_{0,-1}t_{-1,1}t_{1,1}}\\
&&+\mu_0\lambda_0\lambda_1\, {t_{-1,-1}t_{-1,1} t_{2,0}}
+\lambda_{-1}\lambda_0\mu_0\, {t_{1,-1}t_{1,1} t_{-2,0}}
+\lambda_{-1}\lambda_0\lambda_1\, {t_{-1,-1}t_{-1,1} t_{-2,0}}\\
&&+\lambda_0\mu_0(\lambda_0+\mu_0)  \,\frac{t_{-1,-1}t_{-1,1}t_{1,-1}t_{1,1}}{t_{0,0}}\\
\end{eqnarray*}
where the sum extends over the 7 ASM of size 3.
Finally, for $t_{i,j}=1$ for all $i,j$ and $\lambda_i=\mu_i=q^i$ we have:
$$ T_{0,0,3} =4+2(q+q^{-1})=  (1+q^{-1})(1+q^0)(1+q^1) $$
in agreement with Theorem \ref{inladet} for $n=3$.
\end{example}

\subsection{Conclusion and perspectives}

A number of questions remain to be answered concerning this new deformation of the determinant.

The partition function of the 6V-DWBC model at its free fermion point $\Delta=0$, $q=e^{i\pi/2}$, corresponds to
the $2$-enumeration of ASM, with a weight $2$ per entry $-1$, and matches the Lambda-determinant, with all 
$\lambda_a=\mu_b=1$. However, the 6V-DWBC model is known to have a non-trivial deformation involving $2n$
spectral parameters say $z_i$ per row $i$ and $w_j$ per column $j$ of the square lattice $n\times n$ grid.
Our generalized Lambda-determinant introduces $2n-2$ other deformation parameters, which are attached
to diagonals and anti-diagonals of the same square grid. It would be interesting to mix the two deformations
and see if we can obtain some new information on ASMs by the process.

Another question concerns the {\it limit shape}. It is known that the 6V-DWBC model at its free fermion point is also
in bijection with the domino tilings of the Aztec diamond \cite{EKLP}, for which an arctic circle theorem exists,
namely there exists a limiting curve (a circle) separating frozen tiling phases from entropic tiling in the limit 
of large grid and small mesh size. It is easy to work out the homogeneous deformation of this curve
when $\lambda_a=\lambda$ and $\mu_b=\mu$ for all $a,b$. The simplest way of computing the artic curve 
in our model (in the case of a matrix $A$ with entries $a_{i,j}=1$) is
by analyzing the singularities of the generating function of the {\it density} $\rho_{i,j,k}$ defined as
$$ \rho_{i,j,k}=\frac{\partial {\rm Log}\, T_{i,j,k}}{\partial t_{0,0}}\Big\vert_{t_{i,j}=1;i,j\in\Z}$$
where $T_{i,j,k}$ is the solution of the T-system \eqref{inhomtsys} with initial data \eqref{tinit}.
Defining $\rho(X,Y,Z)=\sum_{k\geq 1;i,j\in\Z} \rho_{i,j,k}X^iY^jZ^k$ we find, by differentiating the T-system relation
and evaluating the result at $t_{i,j}=1$:
$$ \rho(X,Y,Z)=\frac{Z}{(1+Z^2-Z(\frac{\lambda}{\lambda+\mu}(X+X^{-1})+\frac{\mu}{\lambda+\mu}(Y+Y^{-1}))} $$
Following \cite{PW}, we find that 
in the coordinates $(x,y)$ obtained as
limit of $(\frac{i}{k},\frac{j}{k})$, for $(i,j)\in D_{i,j,k}$, of a $k\times k$ 6V-DWBC grid rotated by $+\pi/4$, 
the singular curve is the following ellipse:
$$ x^2\left(1+\frac{\lambda}{\mu}\right) + y^2 \left(1+\frac{\mu}{\lambda}\right) =1 $$
It would be interesting to see how this is affected by inhomogeneous parameters $\lambda_a, \mu_a$.
In the particular example $\lambda_i=\mu_i=q^i$ of Section \ref{secq}, we find the following equation
for $\rho(X,Y,Z)$:
$$ (1+Z^2-X-X^{-1})\rho(X,Y,Z)+(1+Z^2-qY-q^{-1}Y^{-1})\rho(q^{-1}X,qY,Z)=2 Z$$
For generic $q$ (not a root of unity) the function $\rho(X,Y,Z)$ is not algebraic, and may be
defined as the following limit:
$$\rho(X,Y,Z)=\frac{2Z}{(1-XZ)(1-X^{-1}Z)} \lim_{t\to 1^-} \sum_{k=0}^{\infty} (-t)^k \prod_{m=1}^{k} 
\frac{(1-q^mYZ)(1-q^{-m}Y^{-1}Z)}{(1-q^{-m}XZ)(1-q^{m}X^{-1}Z)} $$
It would be interesting to investigate its singularity structure.

Finally, the connection to cluster algebra should also allow us to consider quantum versions of the 
T-system, via the quantum cluster algebra construction of \cite {BZ}. A first step in this direction was
taken in \cite{DFK11} for the case of the $A_1$ quantum T-system. Presumably, such quantum deformations 
should yield new non-trivial deformations of the Lambda-determinant and hopefully enrich our understanding
of ASMs.

%\medskip
%\noindent{\bf Acknoledgments.} We would like to thank R. Kedem for many useful discussions.
%We also thank the organizers of the program ``Cluster algebras" and the Mathematical Sciences 
%Research Institute in Berkeley for hospitality during this work.
%We received partial support from the CNRS PICS Grant. 

%\begin{example}
%\end{example}

\end{document}